\documentclass[12pt]{article}
\usepackage[utf8]{inputenc}
\usepackage[margin=1in]{geometry}
\usepackage[svgnames]{xcolor}
\usepackage[bookmarksnumbered=true]{hyperref} 
\hypersetup{
     colorlinks = true,
     linkcolor = blue,
     anchorcolor = blue,
     citecolor = teal,
     filecolor = blue,
     urlcolor = blue
     }
\usepackage{amsmath}
\usepackage{amssymb}
\usepackage{amsthm}
\usepackage{wasysym}
\usepackage{graphicx}
\usepackage{hyperref}
\usepackage{booktabs}
\usepackage{bbold}
\usepackage{color}
\usepackage{setspace}
\usepackage{comment}
\usepackage{multirow}
\usepackage{ulem, tocloft}
\usepackage{tikz}
\usetikzlibrary{intersections}
\usetikzlibrary{arrows}
\usetikzlibrary{cd}
\usetikzlibrary{math}
\usetikzlibrary{calc}
\usetikzlibrary{patterns}
\usetikzlibrary{decorations.markings}
\tikzset{
    segment/.style={line width=1.3pt},
    vertex/.style={
        inner sep=0.7mm,circle,draw,line width=1.5pt,fill=white
    },
    dsegment/.style={segment, dash pattern = on 5pt off 2pt},
    edge/.style={
        line width=1.3pt,
        decoration={
            markings,
            mark=at position 0.6 with {\arrow{>}}
        }
    },
    cylinder/.style = {line width=0.7pt,dash pattern = on 5pt off 5pt},
    graphfill/.style = {gray!50},
}

\theoremstyle{definition}

\newtheorem{remark}{Remark}
\newtheorem{example}{Example}

\theoremstyle{plain}
\newtheorem{theorem}{Theorem}[section]
\newtheorem{corollary}[theorem]{Corollary}
\newtheorem{lemma}[theorem]{Lemma}
\newtheorem{prop}[theorem]{Proposition}

\newtheorem{claim}[theorem]{Claim}

\newcommand{\mL}{\mathcal{L}}

\newcommand{\mV}{\mathcal{V}}
\newcommand{\mE}{\mathcal{E}}

\newcommand{\td}{\mathrm{td}}
\newcommand{\Id}{\mathrm{Id}}

\definecolor{mediumtealblue}{rgb}{0.0, 0.33, 0.71}
\definecolor{caribbeangreen}{rgb}{0.0, 0.8, 0.6}
\definecolor{capri}{rgb}{0.0, 0.75, 1.0}

\newcommand{\isom}{\mathrm{Isom}}

\title{Global Rigidity of Line Constrained Frameworks}
\author{
James Cruickshank 
\and Fatemeh Mohammadi \and Harshit J. Motwani \and Anthony Nixon \and Shin-ichi Tanigawa}
\date{}
\begin{document}

\maketitle

\noindent{\bf Abstract.} 
We consider the global rigidity problem for bar-joint frameworks where each vertex is constrained to lie on a particular line in $\mathbb R^d$. In our setting we allow multiple vertices to be constrained to the same line. 
We give a combinatorial characterisation of generic rigidity in this setting for arbitrary line sets. Further,
under a mild assumption on the given set of lines, we give a complete combinatorial characterisation of graphs that are generically globally rigid. 
This gives a $d$-dimensional extension of the well-known combinatorial characterisation of 2-dimensional global rigidity. In particular, our results imply that global rigidity is a generic property in this setting. 

\medskip
\noindent{}

\section{Introduction}

Consider a discrete geometric structure consisting of a collection of points subject to a system of constraints specifying the distances between points. Such structures are often modelled by graphs with vertices corresponding to points and edges to fixed length line segments. A natural question is to what extent the given structure is unique given the topology of the underlying graph and the specified lengths. The structures are usually referred to as bar-joint frameworks and the question is then the well studied global rigidity problem (see, e.g.,~\cite{JW}).

In particular, while determining global rigidity of a given framework even on the line is NP-hard \cite{Saxe}, a folklore result says that, generically, global rigidity on the line is equivalent to the graph being 2-connected. We will extend this result to higher dimensions giving one of a very small number of combinatorial results for generic global rigidity in arbitrary dimension (see \cite{CJWbb,CJTman,JKTbh} for the others). The general $d$-dimensional case is known to depend only on the underlying graph \cite{GHT} and the case when $d=2$ was resolved in \cite{J&J} but, in general, giving a graph theoretic characterisation is a central challenge in rigidity theory.

Initially, the global rigidity problem arose from attempts to understand configurations of molecules \cite{hendrickson}. More recently global rigidity has found a broader range of applications, for example in localizing networks from partial observations of inter-point distances, see, e.g.~\cite{JJsensor}. In such applications,  it is natural to assume that some points have a fixed location, or perhaps are fixed to move on the ground, or a given wall. Such external boundary conditions can be incorporated into the constraints for a bar-joint framework as additional linear constraints (restriction to move on a fixed hyperplane) on the points \cite{CGJN,GJN,Bill-Tony-Shinichi,Katoh-Tanigawa,ST}. 
In this linearly constrained context, global rigidity has been characterised generically in the plane \cite{GJN} in purely combinatorial terms. However, little is known combinatorially in dimension greater than 2 with, or without, linear constraints. 

In this paper we address the global rigidity problem of bar-joint frameworks whose points are constrained to move on lines. Our research generalises that of \cite{GJN} in two directions. Firstly we will work in arbitrary dimensions and consider bar-joint frameworks constrained to a fixed system of lines. Secondly we will substantially weaken the genericity hypothesis by imposing no restriction on how many points are constrained to any given line. 
This generalization gives us a $d$-dimensional extension of the aforementioned  combinatorial characterisation of $1$-dimensional global rigidity.

The paper is organized as follows.
In Section~\ref{sec:prelim}, we will formally introduce our line constrained rigidity setting and give a characterisation of generic rigidity in this context in Section~\ref{sec:rigidity}. (Rigidity, while fundamental, is a weaker property than global rigidity where uniqueness is only required in a neighbourhood of the given framework.) This characterisation extends, to the general case, a result obtained for line constrained frameworks in \cite{CGJN} since there only 2 points were allowed to lie on each line. We then prove exact necessary conditions for a generic line constrained framework to be globally rigid in Section~\ref{sec:nec}; these results are analogous to the well known Hendrickson conditions for global rigidity of a bar-joint framework \cite{hendrickson}. The remainder of the paper is devoted to showing these necessary conditions are also sufficient, and thereby giving a full combinatorial description of global rigidity. We achieve this by an inductive proof strategy, the geometric aspect of which is to prove that particular graph operations preserve generic global rigidity; this is the content of Section~\ref{sec:exten}. In Section~\ref{sec:small}, we then establish global rigidity for small graphs of three fundamental topological types. These results can then be combined with a combinatorial reduction step, in Section~\ref{sec:main}, to complete the proof of our characterisation.

\section{Preliminaries}

\label{sec:prelim}
\subsection{Line constrained frameworks}
Any line in $\mathbb R^d$ has a unique standard equation of the form $Ax=b$ where $(A,b)$ is a  
$(d-1) \times (d+1)$ matrix in reduced row echelon form. 
Throughout the rest of the paper, $k$ is some positive integer and $\mL = \{L_i: 1 \leq i \leq k \}$ will be a set of lines in $\mathbb R^d$. 
Suppose that $A_i x = b_i$ is the standard representation of the line $L_i$ as above. 
We say that $\mL$ is {\em parallel} if $A_i = A_j$ for all $1 \leq i,j \leq k $ and {\em non-parallel} otherwise.
Let $X$ be the union of the sets of entries of $(A_i,b_i)$ for $1\leq i \leq k$ and 
define $\mathbb Q(\mL)$ to be the smallest subfield of $\mathbb R$ containing $X$. 

A {\em partitioned graph}  is a finite graph $G = (\mV, \mE)$ together 
with a partition of the vertex set $\mV =  \cup_{i =1}^k V_i$. Note that we allow $V_i$ to be empty here. 
Also in contrast to a multipartite graph, the induced edge set on $V_i$ can 
be non-empty. 
We say that a subset $W \subset \mV$ is {\em crossing} if $W \cap V_i$ is non-empty for at least two distinct values of $i$.  We say that a graph or an edge 
is {\em crossing} if its set of vertices is crossing.

Given a line set $\mL$ in $\mathbb R^d$, an {\em $\mL$-constrained framework} is a pair $(G,p)$ where
$G$ is a partitioned graph and $p:\mV \rightarrow \mathbb R^d$ satisfies $p(V_i) \subset L_i$ for $1 \leq i \leq k$.
In this situation, if $u \in V_i$ then define $L_u$ to be $ L_i$.
Let $Y$ be the set of real numbers that appear as a coordinate in some $p(v)$ and define $\mathbb Q(\mL,p)$ to be the smallest subfield of $\mathbb R$ that contains both $\mathbb Q(\mL)$ and $Y$. We say that $(G,p)$ (or $p$) is {\em $\mL$-generic} if 
$\td(\mathbb Q(\mL,p):\mathbb Q(\mL)) = |\mV|$.

\subsection{Isometries on lines}
In the context of framework rigidity, the group of  isometries of the ambient space is important since  such isometries induce trivial 
flexes of a framework.
In this subsection we show basic properties of the isometry group on lines.

Let $\mL$ be a set of lines in $\mathbb{R}^d$.
An {\em isometry of $\mL$} is a map $\theta:\cup_{L \in \mL} L \rightarrow \cup_{L \in \mL} L$, 
such that $|\theta(x) - \theta (y) | = |x-y|$ for
all $x,y \in \cup_{L \in \mL} L$ and such that 
$\theta(L) = L$ for all $ L \in \mL$. In other 
words it is an isometry of the underlying induced metric on $\cup_{L \in \mL}L$ that induces the identity permutation on the set of lines in $\mL$.
Let $\isom(\mL)$ be the group of isometries of $\mL$. 
The following elementary observations will be used repeatedly in the sequel.

\begin{lemma}
\label{lem_twoset}
Suppose that $\mL = \{L_1,L_2\}$ is a set of two lines in $\mathbb{R}^d$, where $L_1, L_2$ are neither parallel nor perpendicular. 
\begin{enumerate}
\item[{\rm (1)}] There is a unique pair $(x_1, x_2)$ of points $x_i\in L_i\ (i=1,2)$ such that $x_i$ is the closest point in $L_i$ to $L_{3-i}$.
Moreover, the coordinates of the point $x_i$ lie in $\mathbb Q(\mL)$. 
\item[{\rm (2)}]  If $\theta \in \isom(\mL)$ is a non-identity element then for $i=1,2$, $\theta|_{L_i}$ is a reflection in $x_i$.
\item[{\rm (3)}]  $\isom(\mL)$ is cyclic of order $2$. 
\end{enumerate}
\end{lemma}
\begin{proof}
Consider the quadratic polynomial $f: L_1 \times L_2 \rightarrow \mathbb R, f(y_1,y_2) = \|y_1-y_2\|^2$.
Since $L_1, L_2$ are not parallel, $f$ is a strictly convex quadratic function and hence has a unique minimiser  $(y_1,y_2) = (x_1,x_2)$.
Therefore $(x_1,x_2)$ is the unique solution to a linear system of equations with coefficients in $\mathbb Q(\mL)$ and statement ${\rm(1)}$ follows from Cramer's rule.

The strict minimality of $f$ implies that any element $\gamma$ of $\isom(\mL)$ must fix $x_1,x_2$. Hence, $\gamma$ restricted to each $L_i$ must be either the identity or the reflection about $x_i$.
The half-turn rotation of $\mathbb{R}^d$ about an axis through $x_1$ and $x_2$ gives a non-identity element $\gamma$ of  $\isom(\mL)$ by restriction.
Also if any element $\gamma$ of $\isom(\mL)$ restricted to $L_i$ is the identity, it must also be the identity on $L_{3-i}$ since $L_1$ and $L_2$ are not perpendicular. These imply (2) and (3). 
\end{proof}

We note also that if $\mL$ is  parallel then 
$\isom(\mL)$ is isomorphic to the Euclidean group of isometries of $\mathbb R$. If $L_1,L_2$ are perpendicular then $\isom(\{L_1,L_2\}) \cong \mathbb Z/2 \times \mathbb Z/2$.
In particular, if $\mL$ is not parallel then $\isom(\mL)$ is
finite.

The situation is more straightforward if the set of lines are non-degenerate in a sense that we now define. 
Three (non-parallel) lines $L_1, L_2, L_3$ in $\mathbb{R}^d$ are {\em weakly concurrent} if 
the closest point  in $L_1$ to $L_2$ coincides with
the closest point  in $L_1$ to $L_3$.
Note that when $d=2$, $L_1, L_2, L_3$ are weakly concurrent if and only if they are concurrent.
We say that a set $\mL$ of lines is in {\em general position} if 
\begin{itemize}
\item no two lines in $\mL$ are parallel,
\item no two lines in $\mL$ are perpendicular, and
\item no three lines in $\mL$ are weakly concurrent.
\end{itemize}

\begin{lemma}
\label{lem_isomgroup}
Suppose that $\mL$ is in general position.
If $|\mL|=1$, then $\isom(\mL)$ is isomorphic to the $1$-dimensional Euclidean group. If $|\mL|=2$ then $\isom(\mL)$ is cyclic of order two and the non-trivial element is a rotation about a common perpendicular of the two lines. If $|\mL|\geq 3$, then $\isom(\mL)$ is trivial. 
\end{lemma}
\begin{proof}
The cases $|\mL| =1,2$ follow immediately from Lemma~\ref{lem_twoset}. Suppose that $|\mL| \geq 3$ and, for a contradiction, suppose that  
$\theta \in \isom(\mL)$ is non-trivial. So there is some
$L_1 \in \mL$ such that $\theta|_{L_1}$ is not the identity.
Choose $L_2,L_3 \in \mL$ such that $L_1,L_2,L_3$ are pairwise distinct. Let $x\in L_1$ be the closest point to $L_2$. By Lemma~\ref{lem_twoset}, $x$ is also the closest point to $L_3$. However this contradicts the fact that $\mL$ is in general position.
\end{proof}

We next state, in a form  suitable for our purposes, one of the fundamental theorems of real algebraic geometry. 
Suppose that $\mathbb F$ is a subfield of $\mathbb R$. Let $X$ be a semi-algebraic subset of $\mathbb R^d$. We say that $X$ is {\em defined over $\mathbb F$} if there is a set of defining polynomial equations and inequalities for $X$ that have coefficients in $\mathbb F$. 

\begin{theorem}[Tarski \cite{Tarski}, Seidenberg \cite{Sei}] 
    Suppose that $\mathbb F$ is a subfield of $\mathbb R$ and that 
    $X \subset \mathbb R^d$ is a semi-algebraic set defined over $\mathbb F$.
    Suppose that $f: \mathbb R^d \rightarrow \mathbb R^c$ is a polynomial map with coefficients in $\mathbb F$. Then 
    $f(X)$ is a semi-algebraic set defined over $\mathbb F$.
    \label{thm_tarskiseidenberg}
\end{theorem}

See \cite[Theorem 2.76]{basu_algo_07} for a proof of this result.  

\subsection{Rigidity of line constrained frameworks}

Given a partitioned graph $G$ and a line set $\mL$ indexed by $J$ as above, let $$\mL^\mV = \{p:\mV \rightarrow 
\mathbb R^d \text{ such that } p(V_i) \subset L_i\text{ for } i \in J\}.$$ 
Observe that $\mL^\mV$ is a $|\mV|$-dimensional linear subspace 
of $\mathbb R^{d|\mV|}$ that is defined by 
equations with coefficients in $\mathbb Q (\mL)$.
Define the 
{\em measurement map} $M_G: \mL^\mV \rightarrow \mathbb R^\mE$ by $M_G(p) = (\|p(u) -p(v)\|^2)_{uv \in \mE}$.
For any subgraph $H$ of $G$, denote $\mL[H] = \{L_i \in \mL: V_i\cap V(H) \neq \emptyset\}$.

We say that $(G,p)$ is {\em globally $\mL$-rigid} if $M_G(q) = M_G(p)$ implies that $q = \theta \circ p$ for some $ \theta \in \isom (\mL[G])$,
and it is {\em $\mL$-rigid} if, for any $q$ in a neighbourhood of $p$ in $\mL^{\mV}$, 
$M_G(q) = M_G(p)$ implies that $q = \theta \circ p$ for some $ \theta \in \isom (\mL[G])$. See Figures~\ref{Example 1 (rigid).pdf} and \ref{Example 2 (P-connected)} for examples.
\begin{figure}[ht]
    \centering
    \includegraphics[width=0.95\textwidth]{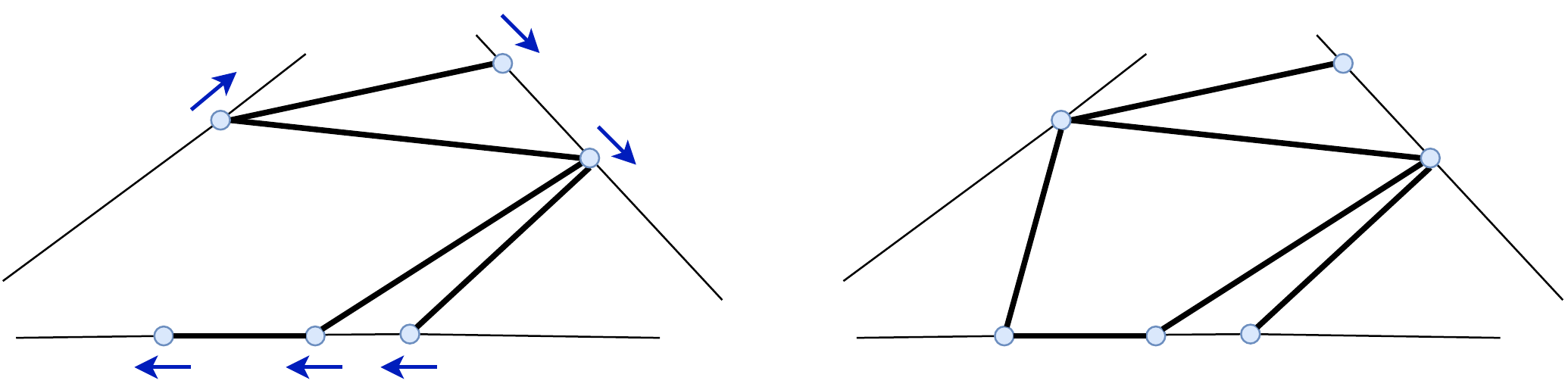}
    \caption{On the left a non-rigid $\mathcal{L}$-constrained framework and on the right an $\mathcal{L}$-rigid framework which is not globally $\mathcal{L}$-rigid.
}
    \label{Example 1 (rigid).pdf}
\end{figure}

We say that a graph $G$ is {\em generically globally $\mL$-rigid} if for any $\mL$-generic $p \in \mL^\mV$, the framework 
$(G,p)$ is globally $\mL$-rigid.
Our main goal is to give a combinatorial characterisation of partitioned graphs $G$ that are globally $\mL$-rigid in the case when $\mL$ is in general position. In particular we shall see that this characterisation is independent of $\mL$, as long as $\mL$ is in general position. Moreover, as a consequence of our characterisation, we will also be able to deduce that the property of global $\mL$-rigidity is a generic property of partitioned graphs. 

For rigidity analysis, it is a common strategy to consider a linearised version known as infinitesimal rigidity.
Suppose that $(G,p)$ is an $\mL$-constrained framework. An {\em infinitesimal 
flex} of $(G,p)$ is a function $f:\mV \rightarrow \mathbb R^d$ satisfying 
\begin{equation}
    \label{eqn_inf_bar_constraint}
    (p(u)-p(v))\cdot (f(u) - f(v)) = 0 \text{ for all $uv \in \mE$ and} 
\end{equation}
\begin{equation}
    \label{eqn_inf_line_constraint}
    A_i f(v) = 0 \text{ for $1 \leq i \leq k, v \in V_i$\ ,}
\end{equation}
where $A_i x = b_i$ is the standard equation of the line $L_i$. 
The coefficient matrix of the linear system defined by (\ref{eqn_inf_bar_constraint}) and (\ref{eqn_inf_line_constraint}), denoted $R(G,p,\mL)$, is called the \emph{$\mL$-rigidity matrix} of $(G,p)$. Thus an 
infinitesimal flex is an element of $\ker(R(G,p,\mL))$.

We say that $(G,p)$ is {\em infinitesimally $\mL$-rigid} if 
\begin{equation}\label{eqn_inf_rigid_def}
\dim(\ker(R(G,p,\mL))) 
=\begin{cases}
1 & (\text{if all lines in $\mL[G]$ are parallel}) \\
0 & (\text{otherwise}).
\end{cases}
\end{equation}
As observed earlier, the right hand side of (\ref{eqn_inf_rigid_def}) is equal to the dimension of $\isom(\mL[G])$, validating our definition of infinitesimal rigidity. 

In general an $\mL$-rigid framework need not be infinitesimally $\mL$-rigid. However 
we will show that for $\mL$-generic frameworks rigidity and infinitesimal rigidity are equivalent. The arguments use a well known technique due to Asimow and Roth in the case of bar-joint frameworks \cite{asi-rot} and so we only sketch the proofs.

\begin{lemma}\label{lem_dimension}
Suppose that $(G,p)$ is $\mL$-generic and let $k=\dim \ker R(G,p,\mL)$.
Then $M_G^{-1}(M_G(p))$ is an algebraic set of dimension $k$. Moreover $p$ has a open neighbourhood $U$ in $\mL^\mV$ such that $M_G^{-1}(M_G(p))\cap U$ is a smoothly embedded submanifold of $U$ of dimension $k$.
\end{lemma}

\begin{proof}[Sketch of proof.]
    First observe that we can view the submatrix of $R(G,p,\mL)$ induced by the rows corresponding to (\ref{eqn_inf_bar_constraint}) as the Jacobian matrix at $p$ of the polynomial function $M_G$. The fact that $p$ is $\mL$-generic ensures that the rank of $R(G,p,\mL)$ is maximal 
    over all points in $\mL^\mV$ which implies the first conclusion. Moreover this rank is constant in an open neighbourhood of $p$. Now the constant rank theorem \cite[Theorem 9]{Spivak1}
    implies the second conclusion.
\end{proof}

\begin{prop}\label{prop:generic_rigidity}
Let $(G,p)$ be a generic $\mL$-constrained framework.
Then, $(G,p)$ is rigid if and only if it is infinitesimally rigid. \qed
\end{prop}

\begin{proof}[Sketch of proof]
The case when $G[\mL]$ is parallel can be deduced from the standard Asimow-Roth theorem for bar-joint frameworks in one dimension \cite{asi-rot}. 
If $G[\mL]$ is not parallel then $(G,p)$ is $\mL$-rigid if and only if there is some neighbourhood $U$ of $p$ in $\mL^\mV$ such that $U \cap M_G^{-1}(M_G(p)) = \{ p\}$. Now the required conclusion follows from Lemma \ref{lem_dimension}.
\end{proof}

In view of this,  we say that $G$ is {\em $\mL$-rigid} if $(G,p)$ 
is infinitesimally 
$\mL$-rigid for all $\mL$-generic $p \in \mL^\mV$. Note that the rank of $R(G,p,\mL)$ is determined by the set of vanishing minors of $R(G,p,\mL)$ and it follows that the maximum rank of $R(G,p,\mL)$ is attained at all $\mL$-generic $p\in \mL^\mV$. 
It follows that $(G,p)$ is infinitesimally $\mL$-rigid for some $p$ if and only if 
$G$ is $\mL$-rigid. 

We conclude this section by deriving another useful property of generic infinitesimally $\mL$-rigid frameworks based on Lemma \ref{lem_dimension}.

\begin{lemma}
    \label{lem_infrigid}
    Suppose that $G[\mL]$ is not parallel and that 
    $(G,p)$ is $\mL$-generic and infinitesimally $\mL$-rigid. Then 
    \begin{enumerate}
        \item[{\rm (1)}] $M_G^{-1}(M_G(p))$ is finite, and
        \item[{\rm (2)}] if $q \in M_G^{-1}(M_G(p))$ and $x$ is a coordinate of $q(v)$ for some
            $v \in \mV$, then $x$ is algebraic over $\mathbb Q(\mL,p)$.
    \end{enumerate}
\end{lemma}

\begin{proof}
    Since $G[\mL]$ is not parallel and $(G,p)$ is infinitesimally $\mL$-rigid, it follows that \break
    $\dim (\ker(R(G,p,\mL))) = 0$.
    The first conclusion now follows from Lemma \ref{lem_dimension}.
    
    To prove the second conclusion, observe that since $M_G^{-1}(M_G(p))$ is a finite semi-algebraic set 
    defined over $\mathbb Q(\mL,p)$,  Theorem~\ref{thm_tarskiseidenberg} implies that, for any $1 \leq i \leq d$, 
    the set of $i$-th coordinates of points in $M_G^{-1}(M_G(p))$ is 
    again a semi-algebraic 
    subset of $\mathbb R$ defined over 
    $\mathbb Q(\mL,p)$. 
    Moreover, since this set is finite, it is 
    a subset of the set of solutions of a nontrivial polynomial equation with coefficients in $\mathbb Q(\mL,p)$. 
    \end{proof}

\section{Characterising Generic \texorpdfstring{$\mL$}{L}-rigidity}\label{sec:rigidity}
In \cite[Theorem 4.3]{CGJN} a combinatorial characterisation of infinitesimal rigidity for line constrained frameworks was given in the case of arbitrary line constraints, but for generic point configurations. 
In \cite{CGJN}, the genericity assumption for point configurations is defined with respect to the ambient Euclidean space $\mathbb{R}^d$, and it is not difficult to see that the result in \cite{CGJN} leads to a characterisation for line constrained frameworks in our sense, 
but only in the case that at most two vertices lie on one line as a consequence of their genericity assumption.

Our goal in this section is to extend this to a characterisation of $\mL$-rigidity for arbitrary line sets $\mL$ with no restrictions on the vertex partition.
This can be accomplished by directly analyzing the rank of the rigidity matrix $R(G,p,\mL)$
of an $\mL$-constrained framework $(G,p)$ of a given graph $G$, unlike the proof in \cite{CGJN}, which is based on an inductive construction.

Let $v_i$ be an unit direction vector of the line $L_i$.
Then any infinitesimal flex $f$ of $(G,p)$ is written as 
$f(u)=t_u v_i$ for some scalar $t_u$ for each $u\in V_i$.
Let $\theta_{u,v}$ be the angle between the vector $v_i$ and $p(v)-p(u)$ for 
each edge $uv\in \mE$ with $u\in V_i$.
Then the linear system (\ref{eqn_inf_bar_constraint}) and (\ref{eqn_inf_line_constraint}) may be rewritten as 
\begin{equation}\label{eqn_inf}
t_u \cos \theta_{u,v} +t_v\cos \theta_{v,u}=0 \quad \text{(for $uv\in \mE$)}.
\end{equation}
The matrix $R'(G,p,L)$ representing (\ref{eqn_inf}) with variables $t_u\ (u\in \mV)$ has size $|\mE|\times |\mV|$, and each column is associated with a vertex and each row is associated with an edge.
Its zero-nonzero pattern is the same as that of the incidence matrix of $G$, and each nonzero entry has the form $\cos \theta_{u,v}$.
Since $\ker R'(G,p,\mL)=\ker R(G,p,\mL)$, it suffices to analyze the rank of $R'(G,p,\mL)$.

\begin{theorem}\label{thm:rigid}
    Let $G$ be a partitioned graph.
    \begin{itemize}
        \item[{\rm (1)}] If $\mL[G]$ is parallel then $G$ is $\mL$-rigid if and only if it 
        is connected.
        \item[{\rm (2)}] If $\mL[G]$ is not parallel then $G$ is $\mL$-rigid if and 
        only if every component of $G$ contains a cycle that contains an edge $uv$ associated with non-parallel lines $L_u$ and $L_v$.
    \end{itemize}
    \label{thm_rigidity_matroid}
\end{theorem}
\begin{proof}
For simplicity, let $\mL=\mL[G]$.
Let $(G,p,\mL)$ be an $\mL$-constrained framework.
We analyse the rank of $R'(G,p,\mL)$.

Suppose first that $\mL$ is parallel, 
i.e., $v_i=v_j$ for every $L_i$ and $L_j$ in $\mL$.
Then, for each $uv\in \mE$,
$\cos \theta_{u,v}=-\cos \theta_{u,v}$.
Also, if $p$ is $\mL$-generic,  $\cos \theta_{u,v}\neq 0$.
Hence, by dividing each row of $R(G,p,\mL)$ associated with $uv$ by $\cos \theta_{u,v}$, $R'(G,p,\mL)$ is converted to the incidence matrix of an edge-oriented digraph of $G$.
So the row matroid of $R'(G,p,\mL)$ is the graphic matroid of $G$, and hence the statement follows in the case when $\mL[G]$ is parallel.

Suppose next that $\mL$ is not parallel.
Our goal is to show that 
$R'(G,p,\mL)$ is non-singular if and only if  every component of $G$ contains a cycle that contains an edge $uv$ such that $L_u$ and $L_v$ are not parallel.
If $G$ is not connected, then  $R'(G,p,\mL)$ is the direct product of the corresponding matrices of the connected components of $G$.
Hence, it suffices to consider the case when $G$ is connected.

For $R'(G,p,\mL)$ to be non-singular, 
it is necessary that $|\mE|=|\mV|$.
Since $G$ is connected, $|\mE|=|\mV|$ implies that 
$G$ contains exactly one cycle $C$.
Then, we can orient each edge $e$ such that 
each vertex has in-degree one in the resulting directed graph.
Let $h(e)$ and $t(e)$ be the head and the tail of an edge $e$ in this orientation.
Then, from the fact that the zero-nonzero pattern of $R'(G,p,\mL)$ is the same as that of the incidence matrix of $G$, the determinant of $R'(G,p,\mL)$ can be expanded as follows:
\[
\det R'(G,p,\mL)=\prod_{e\in \mE\setminus E(C)} \cos \theta_{h(e), t(e)}
\left(
\prod_{e\in E(C)} \cos\theta_{h(e), t(e)}
+(-1)^{|E(C)|-1}\prod_{e\in E(C)} \cos \theta_{t(e), h(e)}
\right).
\]
If $\mL[C]$ is parallel, then $\cos \theta_{h(e),t(e)}=-\cos \theta_{t(e),h(e)}$ for any edge $e$ in $C$, and hence the two terms in the parentheses cancel, i.e.,  
$\det R'(G,p,\mL)=0$.

If $C$ contains an edge $e=uv$ such that $L_u$ and $L_v$ are not parallel, then we consider a point configuration $p$ as follows.
We first put $p(w)$ for all $w\in \mL\setminus \{u\}$ generically on each associated line, and put $p(u)$ such that $p(v)-p(u)$ is orthogonal to $L_u$.
Then $\cos \theta_{u,v}=0$.
Moreover, since $L_u$ is not parallel to $L_v$ and the points are located generically except for $p(u)$,
$\cos \theta_{v,u}\neq 0$ and 
$\cos \theta_{u',v'}\neq 0$ for any other edge $u'v'$.
Thus, among the two terms in the expansion of $\det R'(G,p,\mL)$, exactly one term is nonzero.
Hence $\det R'(G,p,\mL)$ is non-singular.

Therefore, if $G$ is connected, then $\dim \ker R'(G,p,\mL)=0$  if and only if $G$ contains a cycle that has an edge $uv$ associated with non-parallel lines $L_u$ and $L_v$.
\end{proof}


\section{Necessary Conditions for Global \texorpdfstring{$\mL$}{L}-rigidity}
\label{sec:nec}

The rest of the paper is devoted to the line constrained global rigidity 
problem. In that setting the full isometry group of any finite subset of $\mL$ can play a role, so from now on we assume that
\begin{center}
   $\mL$ is in general position. 
\end{center}
When $d=2$, this is equivalent to the assumption that 
no two lines are parallel and perpendicular and 
no three lines are concurrent.

The case in which $G$ is not crossing corresponds to the classical and well understood generic global rigidity problem for frameworks in $\mathbb R^1$. Thus we shall also assume from now on that 
\begin{center}
   $G$ is a {\em crossing graph},  
\end{center}
i.e.~at least two $V_i$ are non-empty. These assumptions will always be in force. However, we will restate them explicitly in the statements of the main results.

We now collect necessary conditions for global $\mL$-rigidity.
We begin with two necessary connectivity conditions.
We say that a connected component of a graph is {\em proper} if it is a proper subgraph.

\begin{lemma}
\label{lem_Pconn0}
Let $(G,p)$ be a $\mL$-generic $\mL$-constrained framework in $\mathbb{R}^d$.
    Suppose that $(G,p)$ is globally $\mL$-rigid. Then, 
    every proper connected component $H$ of $G$ satisfies $|\mL[H]|\geq 3$.
\end{lemma}
\begin{proof}
Suppose there is a proper connected component $H$ of $G$ with $|\mL[H]|\leq 2$.
By Lemma~\ref{lem_isomgroup}, there exists a non-identity element $\gamma\in \isom(\mL[H])$.
Define $q$ such that $q(v)=\gamma p(v)$ for each vertex $v$ in $H$ and $q(v)=p(v)$ for the remaining vertices $v$ of $G$.
Then $M_G(p)=M_G(q)$ holds.
However, since $H$ is a proper subgraph of $G$ and $p$ is $\mL$-generic, $q\neq \theta\circ p$ for any $\theta\in \isom(\mL[G])$.
This contradicts the global $\mL$-rigidity of $(G,p)$.
\end{proof}
\begin{lemma}
\label{lem_Pconn}
Let $(G,p)$ be a $\mL$-generic $\mL$-constrained framework in $\mathbb{R}^d$.
    Suppose that $(G,p)$ is globally $\mL$-rigid. Then, for every $v \in \mV$, every connected component $H$ of $G-v$ satisfies $|\mL[H]|\geq 2$. 
\end{lemma}
\begin{proof}
    Let $H$ be a component of $G-v$.
Suppose, for a contradiction,  that $V(H)\subseteq V_i$ for some $V_i$ in the vertex partition of $G$. 
Let $\gamma$ be the reflection in the hyperplane containing $p(v)$ that is perpendicular to $L_i$. Since $\mL$ is in general position and $p$ is $\mL$-generic,  
the restriction of $\gamma$ to $\bigcup_i L_i$ is not in $\isom(\mL[G])$. 

Let $q(y) = p(y)$ for all $y \in V(G) \setminus V(H)$ and $q(y) = \gamma p(y)$ for $y \in V(H)$. 
Then  $M_G(q) = M_G(p)$ holds.
Hence, by the global $\mL$-rigidity of $(G,p)$, $q = \theta  p$ 
for some $ \theta \in \isom(\mL[G])$. 
Now $p(v) = q(v)$ and since $p$ is $\mL$-generic it follows from Lemma~\ref{lem_twoset} that $\theta = \Id$. Therefore 
$q = p$ and it follows that $\gamma p(y) = p(y)$ for all $ y \in V(H)$ contradicting the 
fact that $p$ is $\mL$-generic since $\gamma$ is a reflection on $L_i$.
\end{proof}

    Next we prove a line constrained version of a well known 
result of Hendrickson~\cite{hendrickson}. 
The standard, and essentially only, known proof of Hendrickson's result for bar-joint frameworks uses Sard's theorem. Our approach is similar but avoids applying Sard's theorem.
Our new technique also works for bar-joint frameworks and we believe it will be useful for rigidity problems in other settings. The following lemma is a key ingredient in our proof.
Recall that a graph is said to be {\em even} if every vertex has even degree.

\begin{lemma}\label{lem:even}
    If $X$ is a bounded real algebraic curve then it is homeomorphic to the geometric realisation of an even graph whose vertices are the singular points of $X$.
    \end{lemma}
    
    \begin{proof}
    Note that $X$ is compact in the usual Euclidean topology, as it is both bounded and closed. Therefore, we have a finite covering $\{N_{\epsilon}(p_i)\}_{i=1,\dots,n}$ of $X$ where $p_i \in X$ and $N_{\epsilon}(p_i)$ is the $\epsilon$-neighbourhood of $p_i$ for sufficiently small $\epsilon > 0$. As real algebraic curves have finitely many singular points \cite[Proposition 3.3.14]{BCR13}, we can chose a finite covering such that it contains all $\epsilon$-neighbourhoods of singular points. Now using \cite[Theorem 9.5.7]{BCR13}, for every $N_{\epsilon}(p_i)$ we have an even number of half-branches centered at $p_i$. From this we can construct a graph $G$ with vertices corresponding to $p_i$ and edges corresponding to the half-branches. Without loss of generality we can assume that we are in a connected component of $X$, which implies that $G$ is connected. Now let $H$ be the graph obtained from $G$ by smoothing out all the vertices of $G$ which are regular points of $X$ so that the vertices of $H$ precisely correspond to the singular points of $X$. In other words, $H$ is a geometric realisation of an even graph with vertices corresponding to the singular points of $X$.  
    \end{proof}

\begin{theorem}
\label{thm:hendrickson}
Suppose that  $p \in \mL^\mV$ is $\mL$-generic and that $(G,p)$ is globally $\mL$-rigid. Then $(G-uv,p)$ is infinitesimally $\mL$-rigid for any edge $uv \in \mE$. 
\end{theorem}

\begin{proof}
   First observe that $(G,p)$ is infinitesimally rigid by Proposition \ref{prop:generic_rigidity}. 
    Hence, Lemma~\ref{lem_infrigid} implies that 
    $\dim(\ker( R(G,p,\mL))) = 0$.
    Similarly, by Lemma~\ref{lem_dimension}, $X = M_{G-uv}^{-1}(M_{G-uv}(p))$ is an algebraic subset  of $\mL^\mV$ and
    $\dim(X) = \dim(\ker(R(G-uv,p,\mL))) \leq \dim(\ker( R(G,p,\mL)))+1 = 1$.
    If $\dim (X) = 0$ then $(G-uv,p)$ is infinitesimally $\mL$-rigid. 
    So we may assume that $\dim(X) = 1$ and so $X$ is an algebraic curve 
    defined over $\mathbb Q(\mL,p)$. 
    
    \begin{claim} $X$ is a bounded subset of $\mL^\mV$. \end{claim}
    \begin{proof} 
        First we make a general observation. Suppose that $H$ is a connected crossing graph and that $r \in \mL^{V(H)}$. For any $w \in V(H)$ there is some $z \in V(H)$ such that $L_w$ and $L_z$ are not parallel. Since there is a path $w = u_1,u_2,\dots,u_k= z$ in $H$, we see that for any $q \in M_H^{-1}(M_H(r))$, $\|q(w) - q(z)\| \leq \sum_{i=1}^{k-1}\|q(u_i) - q(u_{i+1})\| = \sum_{i=1}^{k-1}\|r(u_i) - r(u_{i+1})\|$. Since $L_w$ and $L_z$ are not parallel, it follows that $q(w)$ lies in bounded subset of $L_w$ and so $M_H^{-1}(M_H(r))$ is bounded.
        
        Now suppose that $G'$ is the component of $G$ that contains $uv$ and $G'' = G-G'$. Then $X = Y \times Z$ where
        $Y = M_{G'-uv}^{-1}(M_{G'-uv}(p|_{V(G')}))$ and $Z = M_{G''}^{-1}(M_{G''}(p|_{V(G'')}))$. Since $(G,p)$ is globally $\mL$-rigid, it follows from Lemma \ref{lem_Pconn0} that every component of $G''$ is crossing and so $Z$ is bounded using the general observation above. Thus $X$ is bounded if and only if $Y$ is bounded and so we may as well assume from 
        now that $G = G'$ is a connected graph.
        If $G-uv$ is a connected graph then since it is also a crossing graph the conclusion follows from the general observation above. 

        Hence we can assume that $G-uv$ is not connected and (since $G$ is connected) has two components $H$ and $K$. Suppose 
        $H$ and $K$ contain the vertices $u$ and $v$ respectively. If $|\mL[H]| = 1$ then, since $p$ is generic (in particular $p(u) - p(v)$ is not perpendicular 
        to $L_u$), we can slide $(H,p|_{V(H)})$ along $\mL[H]$ to find $(H,p')$ such that 
        $\|p'(u) - p(v)\| = \|p(u) - p(v)\|$ and so that $p' \neq p|_{V(H)}$, contradicting 
        the global rigidity of $(G,p)$. Thus $|\mL[H]| \geq 2$, and similarly $|\mL[K]| \geq 2$. 
        Let $Y = \{q \in \mL^{V(H)}: M_{H}(q) = M_H(p|_{V(H)})\}$ and $Z = \{q \in \mL^{V(K)}: M_{K}(q) = M_K(p|_{V(K)})\}$.
        Using the general observation we see that both $Y$ and $Z$ are 
        bounded. Now, $X \subset Y \times Z$ and the claim follows.
    \end{proof}
    
    Thus $X$ is a bounded real algebraic curve, and so by Lemma~\ref{lem:even}, it is homeomorphic to the geometric realisation of an even graph whose vertices are singular points of $X$. Also $p$, being $\mL$-generic, is a non-singular point of $X$ by Lemma~\ref{lem_dimension} and therefore lies in the interior of an edge of $X$ (viewed as a graph).
    
    In the case where $|\mL[G]|\geq 3$, $\isom(\mL[G])$ is trivial by Lemma~\ref{lem_isomgroup}.
    In the case where $|\mL[G]| =2$, let $\theta$ be the unique non-trivial element of $\isom(\mL[G])$. 
    We show that $\theta$ acts without fixed points on $X$. 
    To see this, suppose  $\theta$ fixes $q\in X$.
    Then by Lemma~\ref{lem_twoset} 
    $\theta|_{L_i}$ is the reflection around 
    $x_i$ for each $i=1,2$, where $x_i$ is the closest point in $L_i$ to $L_{3-i}$.
    So $\theta q=q$ implies that each point of $q$ lies in $x_1$ or $x_2$, and $\|q(w)-q(z)\|^2$ is equal to either $\|x_1-x_2\|^2$ or $0$. 
    However, since $x_1$ and $x_2$ are algebraic over $\mathbb{Q}(\mL[G])$,
    $\|p(w)-p(z)\|^2$ would be algebraic over $\mathbb{Q}(\mL[G])$.
    This contradicts the $\mL$-genericity of $p$. Therefore $\theta$ induces
    a fixed point free involution of $X$ in the case $|\mL[G]| = 2$. 
    
    Thus, in all cases, $X/\isom(\mL[G])$ is homeomorphic to the geometric realisation of an even graph. 
    Let $\overline p$ be the image of
    $p$ in $X/\isom(\mL[G])$.
   Since $p$ is $\mL$-generic, Lemma \ref{lem_dimension} implies that a neighbourhood of $p$ in $X$ is homeomorphic to $\mathbb R$ and thus a neighbourhood of $\overline p$ in $X/\isom(\mL[G])$ is homeomorphic to $\mathbb R$. In particular this implies that $\overline p$ is either an interior point of an edge or a vertex of degree 2 in $X/\isom(\mL[G])$ (viewed as a graph).
    Let $Y$ be the component of $X/\isom(\mL[G])$ that contains $\overline p$. Since a connected even graph is 
    2-edge-connected it follows that $Y-\overline p$ is connected. Now 
    define $f :Y \rightarrow \mathbb R$ by $f(\overline q) =  \|q(u) - q(u)\|^2$. Since $(G,p)$ is
    rigid
    and $(G-uv,p)$ is not rigid, Proposition \ref{prop:generic_rigidity} implies that $(G,p)$ is infinitesimally rigid and $(G-uv,p)$ is not. 
    Hence $f'(p) \neq 0$. 
    So there are points $q_1,q_2 \in Y$
    such that $f(\overline{q}_1) < f(\overline p) < f(\overline{q}_2)$. Now, since $Y-\overline p$ is connected $Y-\overline{p}$ has a path between $\overline{q}_1$ and $\overline{q}_2$, and the path contains  
    a point $p' \in X$ such that $\overline{p}' \neq \overline{p}$ and $f(\overline{p}') = f(\overline{p})$ contradicting the fact that 
    $(G,p)$ is globally $\mL$-rigid.
\end{proof}

We say that a partitioned graph $G = (\mV,\mE)$ is {\em P-connected} if 
every proper connected component $H$ of $G$ satisfies $|\mL[H]|\geq 3$
and every connected component $H'$ of $G-v$ satisfies $|\mL[H']|\geq 2$ for every $v\in \mV$. See Figure~\ref{Example 2 (P-connected)} (left) for an example. 
Note that a P-connected graph is not necessarily connected.

We say that $G$ is {\em redundantly  $\mL$-rigid} if $G-e$ is $\mL$-rigid for all $e \in \mE$. Observe that in the case that $\mL$ is in general position and $G$ is crossing, Theorem~\ref{thm_rigidity_matroid} implies that $G$ is redundantly rigid if and only if, for every $e \in \mE$, every component of $G-e$ contains a crossing cycle. See Figure~\ref{Example 2 (P-connected)} (right). Lemma~\ref{lem_Pconn0}, Lemma~\ref{lem_Pconn} and Theorem~\ref{thm:hendrickson} imply the following necessary conditions for a partitioned graph to be generically globally $\mL$-rigid. 

\begin{figure}[ht]
    \centering
    \includegraphics[width=0.95\textwidth]{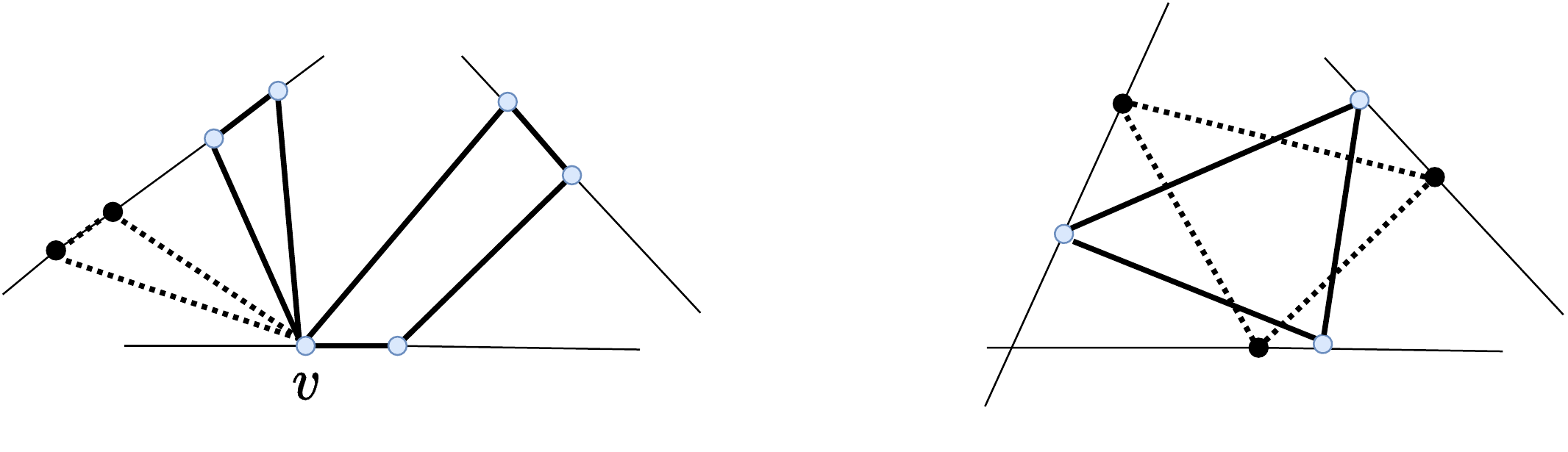}
    \caption{Examples of $\mathcal{L}$-frameworks which are not globally $\mathcal{L}$-rigid. 
    On the left the underlying graph is redundantly $\mathcal{L}$-rigid but not P-connected.
    On the right the underlying graph is P-connected but 
    not redundantly $\mathcal{L}$-rigid.
    In both cases, a global flex is indicated with dashed bars.
}
    \label{Example 2 (P-connected)}
\end{figure}

\begin{theorem}
    \label{thm_nec}
    Suppose that $\mL$ is in general position and $G$ is a crossing partitioned graph. If $G$ is generically globally $\mL$-rigid then $G$ is P-connected and redundantly $\mL$-rigid. 
\end{theorem}

\section{Inductive Properties}
\label{sec:exten}

Our goal is to prove the converse of Theorem~\ref{thm_nec}. The basic strategy will be 
induction on the number of vertices of the graph and in
this section we will analyse the geometric properties of the relevant inductive operations.
For a graph $G$, a {\em subdivision} of an edge $uv$ replaces $uv$ with new edges $uw$ and $wv$ by adding a new vertex $w$ of degree two.
The inverse operation is called {\em smoothing} at $w$.
We will show that the subdivision operation preserves generic global $\mL$-rigidity. We begin with an elementary geometric lemma.

\begin{lemma}
\label{lem_geom1}
Suppose that $L$ is a line in $\mathbb R^d$ with equation $Ax=b$ and that $p_1,p_2$ are distinct points in $\mathbb R^d$. Let $\mathbb F$ be the smallest subfield of $\mathbb R$ that contains all the entries of $A,b,p_1,p_2$. Let $f:L \rightarrow \mathbb R^2$ be given by $f(x) = (\|x-p_1\|^2,\|x-p_2\|^2)$. Then the following hold:
\begin{enumerate}
    \item[{\rm (1)}] If the line segment $[p_1,p_2]$ is perpendicular to $L$ then $f(L)$ is a half-line defined over $\mathbb F$.
    \item[{\rm (2)}] If $[p_1,p_2]$ is not perpendicular to $L$ then $f(L)$ is a parabola that is defined over $\mathbb F$. In particular $f(L)$ is an irreducible algebraic set defined over $\mathbb F$ in this case.
\end{enumerate}
\end{lemma}

\begin{proof}
Let $\pi:\mathbb R^d \rightarrow L$ be the orthogonal projection. Then 
$$f(x) =
(\|x - \pi(p_1)\|^2, \|x-\pi(p_2)\|^2) + (\|p_1 -\pi(p_1)\|^2,\|p_2-\pi(p_2)\|^2).$$ Now if $\pi(p_1) =\pi(p_2)$ then $\{(\|x - \pi(p_1)\|^2, \|x-\pi(p_2)\|^2), x \in L\} = 
\{(a,a): a \geq 0\}$, whereas it is elementary to check that 
if $\pi(p_1) \neq \pi(p_2)$ then $\{(\|x - \pi(p_1)\|^2, \|x-\pi(p_2)\|^2), x \in L\}$ is a parameterisation of a parabola.
\end{proof}

\begin{example}\label{exm:lemma5_1}
In order to explain the intuition behind Lemma \ref{lem_geom1}, we can work with $d=2$ and choose local coordinates on a given line $L$ such that it is represented by the $x$-axis i.e., $L = \{(t,0): t \in \mathbb{R}\}$.
The two parts of the lemma are explained as follows:
\begin{enumerate}
    \item If the line segment $[p_1,p_2]$ is perpendicular to $L$, then $f(L)$ is a half-line. Without loss of generality we can choose $p_1 = (0,1)$ and $p_2 = (0,-1)$. Then $f(L)$ is given by $\{(t^2+1 ,t^2 +1): t \in \mathbb{R}\}$ which describes a half line in $\mathbb{R}^2$ starting at $(1,1)$ with a slope of $1$ from $x$-axis, as depicted in Figure~\ref{fig_lemma_5_1} (left).
    \item If the line segment $[p_1,p_2]$ is not perpendicular to $L$, then $f(L)$ is a parabola. Similar to the above case we can choose $p_1 = (0,1)$ and $p_2 =  (1,0)$, where $f(L)$ is given by $\{(t^2+1 ,t^2 +1 - 2t): t \in \mathbb{R}\}$. This is a parametric equation of parabola whose major-axis is rotated from the reference axis. See Figure~\ref{fig_lemma_5_1} (right).
    
\end{enumerate}

\begin{figure}[ht]
    \centering
    \includegraphics[width=0.4\textwidth]{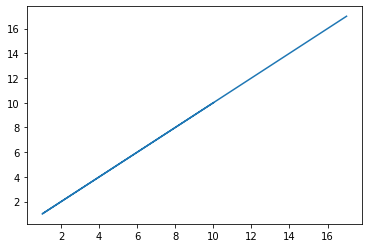}
    \includegraphics[width=0.4\textwidth]{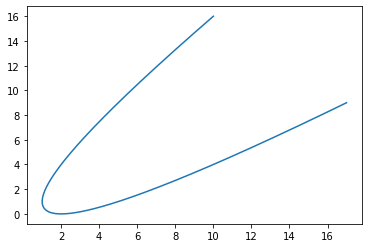}
    \caption{Figures for Example \ref{exm:lemma5_1}.}
    
    \label{fig_lemma_5_1}
    \end{figure}
\end{example}

We also remark the following elementary fact since it is used frequently in the subsequent discussion.
\begin{lemma}
\label{lem_symmetry}
Let $L_1, L_2$ be lines in $\mathbb{R}^d$ in general position, $\pi:\mathbb{R}^d\rightarrow L_2$ be the orthogonal projection to $L_2$, and let $f:L_1\rightarrow \mathbb{R}$ be given by $f(x)=\|x-\pi(x)\|^2$.
Then $f$ is a strictly convex function,
whose minimum is attained at the closest point to $L_2$
and which is symmetric with respect to the closest point.
\end{lemma}
Lemma~\ref{lem_symmetry} and Lemma~\ref{lem_twoset} in particular imply that,
if $x_1$ and $x_2$ are two points on $L_1$ such that 
$\|x_1-\pi(x_1)\|=\|x_2-\pi(x_2)\|$,
then there is $\theta\in \isom(\{L_1, L_2\})$ such that 
$x_2=\theta x_1$.

The following technical lemma is a key observation in the proof of Theorem \ref{thm_main_henn}.

\begin{lemma}
\label{lem_tech}
Let $(G,p)$ be an $\mL$-generic $\mL$-constrained framework,
and let $u, v, w$ be distinct vertices of $G$ such that 
$uw$ and $vw$ are edges of $G$.
Suppose that $(G-w,p|_{\mV-w})$ is infinitesimally $\mL$-rigid. 
Then, for each $q \in M_{G}^{-1}(M_G(p))$,
we have $\|p(u)-p(v)\|=\|q(u)-q(v)\|$.
Moreover, at least one of the following holds:
\begin{itemize}
    \item $L_u=L_w$, 
    \item $L_v=L_w$, or 
\item  there exists $\gamma \in \isom(\{L_u, L_v,L_w\})$ such that $q(x) = \gamma p(x)$ for $x \in \{u,v,w\}$.
\end{itemize}
\end{lemma}
\begin{proof}
Pick any $q\in M_G^{-1}(M_G(p))$,
and define $f:L_w \rightarrow \mathbb R^2$ by $f(x) = (\|x-p(u)\|^2, \|x-p(v)\|^2)$ and 
$g:L_w \rightarrow \mathbb R^2$ by $g(x) = (\|x-q(u)\|^2, \|x-q(v)\|^2)$. 
Since $p$ is generic 
it follows from Lemma~\ref{lem_geom1} that $f(L_w)$ is a non-linear irreducible algebraic curve defined over $\mathbb Q(\mL,p|_{\mV- w})$. 
On the other hand, by Lemma~\ref{lem_geom1}, $g(L_w)$ is either an irreducible algebraic curve over $\mathbb Q(\mL,q|_{\mV-w})$ or is contained in a line in $\mathbb R^2$ that is defined over 
$\mathbb Q(\mL,q|_{\mV-w})$. 
Moreover, by Lemma~\ref{lem_infrigid}, $\mathbb Q(\mL,q|_{\mV-w})$ is contained in the algebraic closure of $\mathbb Q(\mL,p|_{\mV- w})$. 
Hence, $g(L_w)$ is also defined over $\mathbb Q(\mL,p|_{\mV- w})$.
Now observe that $f(p(w)) \in f(L_w) \cap g(L_w)$. Since at least one coordinate of $p(w)$ is transcendental over $\mathbb Q(\mL,p|_{\mV- w})$, at least one coordinate of $f(p(w))$ is transcendental over $\mathbb Q(\mL,p|_{\mV- w})$.
Since  each point in $f(L_w) \cap g(L_w)$ would be algebraic over $\mathbb Q(\mL,p|_{\mV- w})$ if it were 0-dimensional,  it follows that $f(L_w) \cap g(L_w)$ is not 0-dimensional.
Since $f(L_w)$ is irreducible, it further implies  that 
$f(L_w) = g(L_w)$.

Let $\pi:\mathbb R^d \rightarrow L_w$ be the orthogonal projection. 

   \begin{claim}
    \label{clm_geometric1}
        It follows that 
        \begin{enumerate}
            \item[{\rm (a)}] $\|p(u) - \pi(p(u))\| = \|q(u) - \pi(q(u))\|$ and  
                $\|p(v) - \pi(p(v))\| = \|q(v) - \pi(q(v))\|$, and
            \item[{\rm (b)}] $\|\pi(p(u)) - \pi(p(v))\| = \|\pi(q(u)) - \pi(q(v))\|$.
        \end{enumerate}
    \end{claim}
    \begin{proof}[Proof of claim.]
        Observe that 
        \begin{align*}
        \|p(u) -\pi(p(u))\| &= \min\{\sqrt x: (x,y) \in f(L_w)\} \\
        &= \min\{\sqrt x: (x,y) \in g(L_w)\} = \|q(u) - \pi(q(u))\|,
        \end{align*}
        where the second equation follows from $f(L_w)=g(L_w)$. This proves (a). 
        For (b), we have
        \begin{align*}
        &\|\pi(p(u)) - \pi(p(v))\| \\&= \min\{ \sqrt{x -\|p(u) -\pi(p(u))\|^2} +\sqrt{y-\|p(v) -\pi(p(v))\|^2}: (x,y) \in f(L_w)\}, \\
        &\|\pi(q(u)) - \pi(q(v))\| \\&= \min\{ \sqrt{x  - \|q(u) -\pi(q(u))\|^2} +\sqrt{y-\|q(v) -\pi(q(v))\|^2}: (x,y) \in g(L_w)\}. 
        \end{align*}
        Using (a) and $f(L_w)=g(L_w)$, (b) follows. 
    \end{proof}
By Claim~\ref{clm_geometric1}(a) and Lemma~\ref{lem_symmetry}, at least one of the followings hold for vertex $u$:
\begin{itemize}
\item $L_u=L_w$,
\item $p(u)=q(u)$, or
\item $p(u)=\theta q(u)$, where $\theta$ is the nontrivial element in $\isom(\{L_u, L_w\})$.
\end{itemize}
Indeed, if $L_u\neq L_w$ and $p(u)\neq q(u)$, then Claim~\ref{clm_geometric1}(a) and Lemma~\ref{lem_symmetry} imply that $p(u)$ and $q(u)$ are symmetric with respect to the closest point of $L_u$ to $L_w$, and hence 
$p(u)=\theta q(u)$ for the nontrivial element  $\theta$ in $\isom(\{L_u, L_w\})$.
The corresponding property also holds for vertex $v$.

Suppose $L_u=L_w$.
By Claim~\ref{clm_geometric1},
$\|p(v)-\pi(p(v))\|=\|q(v)-\pi(q(v))\|$
and $\|p(u)-\pi(p(v))\|=\|q(u)-\pi(q(v))\|$ hold,
and by Pythagoras theorem we get 
$\|p(u)-p(v)\|=\|q(u)-q(v)\|$ as required.

Symmetrically, the statement holds if $L_v=L_w$.
Hence, in the subsequent discussion, we assume $L_u\neq L_w\neq L_v$.

Suppose that $q(u) = p(u)$ and $q(v) = p(v)$. Then 
$\|p(u)-p(v)\|=\|q(u)-q(v)\|$ clearly holds,
and it remains to show the existence of an isometry $\gamma\in \isom(\{L_u, L_v, L_w\})$  as in the statement. If $q(w) = p(w)$ then this  is true (by taking $\gamma = \Id$), so assume $q(w) \neq p(w)$. Since $M_G(q) = M_G(p)$ it follows that $L_w$ is perpendicular to the line through $p(u)$ and $p(v)$,
which contradicts the  fact that $\mL$ is in general position and that $p$ is $\mL$-generic. 

Thus, without loss of generality, we can assume that $q(v) \neq p(v)$. Then, $q(v) = \gamma p(v)$ holds for  the non-trivial isometry $\gamma$ of $\{L_v,L_w\}$. In particular $y = (\pi(p(v))+\pi(q(v)))/2$ is the point on $L_w$ that is closest to $L_v$.
If $q(u) = p(u)$ then it follows from 
 Claim \ref{clm_geometric1}(b) that $\pi(p(u)) = y$, 
contradicting the fact that $p$ is $\mL$-generic. Therefore $q(u) \neq p(u)$, and hence  $q(u) = \tau p(u)$ holds for the non-trivial isometry $\tau$ of $\{L_u,L_w\}$.
In particular, $z = (\pi(p(u)) + \pi(q(u)))/2$ is the closest point on $L_w$ to $L_u$. 

Now by Claim~\ref{clm_geometric1}(b) we have $\pi(p(u)) - \pi(p(v)) = \pm (\pi(q(u)) - \pi(q(v)))$. 
If $\pi(p(u)) - \pi(p(v)) = \pi(q(u)) - \pi(q(v))$ then 
$$\pi(p(u)) - z = \frac{1}{2}(\pi(p(u)) - \pi(q(u))) = \frac{1}{2}(\pi(p(v))-\pi(q(v))) = \pi(p(v)) - y.$$ 
However, this contradicts the fact that $p$ is $\mL$-generic since the points $y,z$ have coordinates that lie in $\mathbb Q(\mL)$.

Therefore $\pi(p(u)) - \pi(p(v)) = -(\pi(q(u)) - \pi(q(v)))$ and so $y = z$. Since $\mL$ is in general position (in particular no three lines are weakly concurrent), this implies that $L_u = L_v$ and so $\tau = \gamma$ is the non-trivial isometry of $\{L_u,L_v,L_w\}$.
Hence, $\|p(u)-p(v)\|=\|q(u)-q(v)\|$ follows. 

Now suppose that $q(w) \neq \gamma p(w)$. Then it follows that $L_w$ is perpendicular to $L_u$ contradicting the fact that $\mL$ is in general position. Therefore $q(x) = \gamma p(x)$ for all $x \in \{u,v,w\}$. This completes the proof.
\end{proof}

We are ready to show that global $\mL$-rigidity is preserved by subdivision.

\begin{theorem}
\label{thm_main_henn}
Let $(G,p)$ be an  $\mL$-generic $\mL$-constrained framework and  $w$ be a vertex of degree two. 
Let $G'$ be the graph obtained from $G$ by smoothing at $w$,
and let $p':\mV-w \rightarrow \mathbb R^d$ be the restriction of $p$. If $(G',p')$ is globally $\mL$-rigid, then $(G,p)$ is globally $\mL$-rigid.
\end{theorem}

\begin{proof}
Let $u,v$ be the neighbors of $w$ in $G$.
By Theorem~\ref{thm_nec}, $G$ is redundantly $\mL$-rigid.
Since $w$ is degree two in $G$, $G-w$ is $\mL$-rigid
and $(G-w,p')$ is infinitesimally $\mL$-rigid.

Let $q\in M_G^{-1}(M_G(p))$.
By Lemma~\ref{lem_tech}, $\|p(u)-p(v)\|=\|q(u)-q(v)\|$.
Since $(G',p')$ is globally $\mL$-rigid
it follows that $q'=\theta p'$ for some  $\theta \in \isom(\mL[G-w])$.
We split the proof into two cases.

Suppose $\isom(\mL[G])=\isom(\mL[G-w])$.
Then $\theta \in \isom(\mL[G-w])=\isom(\mL[G])$. Now, replacing $q$ by $\theta^{-1} q$ if necessary, we can assume that $q'=p'$. Therefore, since $q \neq p$,  $q(w) \neq p(w)$.
    Now since $\|p(u) - p(w)\| = \|q(u) - q(w)\| = \|p(u) - q(w)\|$ and similarly $\|p(v) -p(w) \| = \|p(v) - q(w)\|$, the line
    through $p(u)$ and $p(v)$ is perpendicular to $L_w$. This contradicts the facts that 
    $\mL$ is in general position  and $p$ is $\mL$-generic. 

Next, suppose that $\isom(\mL[G])\neq \isom(\mL[G-w])$.
$G$ is P-connected by Theorem~\ref{thm_nec},
 so $|\isom(\mL[G-w])|\geq 2$.
By Lemma~\ref{lem_isomgroup}, we have 
$|\isom(\mL[G-w])|=2$ and $|\isom(\mL[G])|=3$.
Hence $L_u\neq L_w\neq L_v$ follows.
So by Lemma~\ref{lem_tech}, there exists $\gamma\in \isom(\{L_u, L_v, L_w\})$ such that 
$q(x)=\gamma p(x)$ for $x\in \{u,v,w\}$.
If $\theta$ is the identity, 
then $p(u)=q(u)$ follows, and hence $\gamma$ is the identity as well. Then $p=q$ follows.

If $\theta$ is not the identity, 
then $\gamma$ is also not the identity.
Denote 
$\mL[G-w]=\{L_1, L_2\}$ and without loss of generality suppose $L_u=L_1$.
Then $\theta|_{L_1}$ is the reflection about the closest point $y$ to $L_2$ while $\gamma|_{L_1}$ is the reflection about the closest point $z$ to $L_w$.
By $\gamma p(u)= q(u) = \theta p(u)$, $y=z$ follows, contradicting that $L_1, L_2, L_w$ are not weakly concurrent. This completes the proof.
\end{proof}

Next we prove a gluing property for generic global $\mL$-rigidity.

\begin{theorem}
\label{thm_gluing}
Suppose that $H,K$ are crossing partitioned graphs that  are both globally $\mL$-rigid.
If $V(H) \cap V(K)$ is non-empty 
or $|\mL[H]|\geq 3$ and $|\mL[K]|\geq 3$, then  $H \cup K$ is generically globally $\mL$-rigid. 
\end{theorem}

\begin{proof}
Suppose $V(H) \cap V(K)$ is non-empty.
Choose $u \in V(H) \cap V(K)$. Let $p,q\in \mL^{V(H) \cup V(K)}$ with $p$ $\mL$-generic and $M_{H\cup K}(q) = M_{H \cup K}(p)$. Since $H,K$ are generically
globally $\mL$-rigid, it follows that there exist $\theta \in \isom (\mL[H])$ and $\gamma \in \isom(\mL[K])$ such that $q(v) = \theta p(v)$ for $v \in V(H)$ and 
$q(w) = \gamma p(w)$ for $w \in V(K)$. If $\theta = \Id$ then since $u \in V(H) \cap V(K)$ we see that $\gamma p(u) = q(u) = \theta p(u) = p(u)$ and since $p$ is generic it follows that $\gamma = \Id$ and so $q = p$. Similarly if $\gamma = \Id$ then $q = p$. So we can assume that $\theta, \gamma$ are both non-trivial isometries. Since $\theta p(u) = \gamma p(u)$, Lemma~\ref{lem_isomgroup} implies  that $\mL[H] = \mL[K] = \mL[H \cup K] $ and that $\theta = \gamma \in \isom (\mL[H \cup K])$ and that $ q = \theta p$. 

Suppose $V(H) \cap V(K)$ is empty.
Then $|\mL[H]|\geq 3$ and $|\mL[K]|\geq 3$ hold,
and both $\isom(\mL[H])$ and $\isom(\mL[K])$ are trivial by Lemma~\ref{lem_isomgroup},
and the global rigidity of $H\cup K$ follows from that of $H$ and $K$.
\end{proof}

Note that the hypothesis that both $H$ and $K$ are crossing is required for Theorem~\ref{thm_gluing} to be valid.

\section{Basic Globally Rigid Graphs}
\label{sec:small}

In this section we establish generic global rigidity 
for certain classes of graphs which will be used as base cases in an inductive argument for our main result in the next section.
As usual $G$ is a partitioned graph and $\mL$ is in general position.
We begin with the case when $G$ is a triangle.

\begin{lemma}
\label{lem_triangle_flex}
Suppose that $T = (\mV, \mE)$ is the partitioned graph where $\mV$ is the disjoint union of $V_1$ and $V_2$, $V_1 = \{u,v\}, V_2 = \{w\}$ and $\mE = \{uv,vw,wu\}$. Let $p,q\in \mL^{\mV}$ with $p$ $\mL$-generic and suppose that $M_{T}(q) = M_T(p)$. Then $q(w) = \theta p(w)$ for some $\theta \in \isom(\mL[T])$.
\end{lemma}
\begin{proof}
Let $\pi:\mathbb R^d \rightarrow L_u$ be the orthogonal projection and observe that, since the triangles $p(u),p(v),p(w)$ and $q(u),q(v),q(w)$ are congruent and $p(u),p(v),q(u),q(v) \in L_1$, we have that $|q(w) -\pi(q(w))| = |p(w) - \pi(p(w))|$. It follows from Lemma~\ref{lem_symmetry} that $q(w) = \theta p(w)$ for some $\theta \in \isom (\{L_w,L_u\})$ as required.
\end{proof}

Next we introduce three basic types of partitioned graphs
that will occur in our inductive argument in the next 
section. We refer to Figure~\ref{fig_types} for illustrations.
A partitioned graph $G$ is of 
\begin{itemize}
    \item {\em type 1} if it is a union of two cycles that have a single vertex $v$ in common, and, each of the cycles contains a crossing edge that is not incident to $v$. 
    \item {\em type 2} if it is a union of two disjoint cycles $C_1, C_2$ and a simple path $P$ with endpoints $v_1,v_2$ such that $\{v_i\} = V(P) \cap V(C_i), i =1,2$, and, such that for $i =1,2$, $C_i$ contains a crossing edge that is not incident to $v_i$.
    \item {\em type 3} if it is a union of a cycle $C$ and a simple path $P$ with endpoints $v_1,v_2 \in V(C)$ such that $V(P) \cap V(C) = \{v_1,v_2\}$ and there are crossing edges $e,f$ such that $e,f$ are not adjacent and so that each simple path of $C$ joining $v_1$ to $v_2$ contains one of $e,f$. 
\end{itemize}

\begin{figure}
    \centering
    \begin{tabular}{ccc}
        \begin{tikzpicture}
        \clip (-2.2,-1) rectangle (2.2,1);
    \node[vertex] (a) at (0,0){};
    \node[vertex] (b) at (-2,-0.6){};
    \node[vertex] (c) at (-2,0.6){};
    \node[vertex] (d) at (2,0.6){};
    \node[vertex] (e) at (2,-0.6){};
    \draw[edge,dash pattern=on 2pt off 4pt] (a) edge [bend right] (c);
    \draw[edge,dash pattern=on 2pt off 4pt] (a) edge [bend left] (b);
    \draw[edge,dash pattern=on 2pt off 4pt] (a) edge [bend left] (d);
    \draw[edge,dash pattern=on 2pt off 4pt] (a) edge [bend right] (e);
    \draw[edge] (b) -- (c);
    \draw[edge] (d) -- (e);
\end{tikzpicture}
& 
\begin{tikzpicture}
        \clip (-2.2,-1) rectangle (3.2,1);
    \node[vertex] (a) at (-0.2,0){};
    \node[vertex] (aa) at (1.2,0){};
    \node[vertex] (b) at (-2,-0.6){};
    \node[vertex] (c) at (-2,0.6){};
    \node[vertex] (d) at (3,0.6){};
    \node[vertex] (e) at (3,-0.6){};
    \draw[edge,dash pattern=on 2pt off 4pt] (a) edge [bend right]  (c);
    \draw[edge,dash pattern=on 2pt off 4pt] (a) edge [bend left] (b);
    \draw[edge,dash pattern=on 2pt off 4pt] (aa) edge [bend left] (d);
    \draw[edge,dash pattern=on 2pt off 4pt] (aa) edge [bend right] (e);
    \draw[edge,dash pattern=on 2pt off 4pt] (aa) -- (a);
    \draw[edge] (b) -- (c);
    \draw[edge] (d) -- (e);
\end{tikzpicture}
&
\begin{tikzpicture}
        \clip (-2.2,-1) rectangle (2.2,1);
    \node[vertex] (a) at (0,0.7){};
    \node[vertex] (aa) at (0,-0.7){};
    \node[vertex] (b) at (-2,-0.6){};
    \node[vertex] (c) at (-2,0.6){};
    \node[vertex] (d) at (2,0.6){};
    \node[vertex] (e) at (2,-0.6){};
    \draw[edge,dash pattern=on 2pt off 4pt] (a) edge [bend right] (c);
    \draw[edge,dash pattern=on 2pt off 4pt] (aa) edge [bend left] (b);
    \draw[edge,dash pattern=on 2pt off 4pt] (a) edge [bend left] (d);
    \draw[edge,dash pattern=on 2pt off 4pt] (aa) edge [bend right] (e);
    \draw[edge,dash pattern=on 2pt off 4pt] (aa) -- (a);
    \draw[edge] (b) -- (c);
    \draw[edge] (d) -- (e);
\end{tikzpicture}
 \\
 Type 1 & Type 2 & Type 3

    \end{tabular}

    \caption{The three basic types of graphs. The solid edges are crossing. The dotted arcs represent simple paths in the graph. In the case of the type 1 or type 2 graph, all of the dotted paths must have at least one edge. In the case of the type 3 graph, some of the dotted paths may have length zero as long as the vertical dotted path  in the centre has at least one edge and the two solid edges are vertex disjoint.}
    \label{fig_types}
\end{figure}
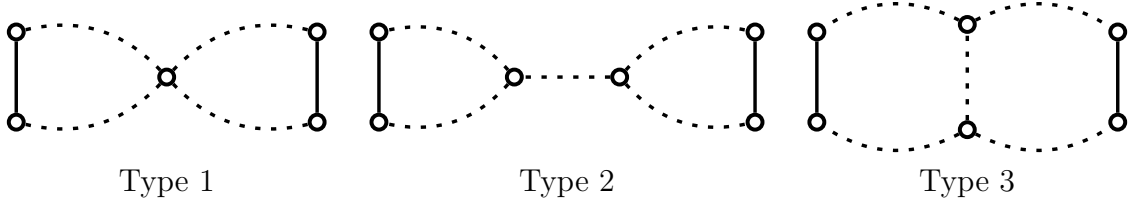

\begin{lemma}
\label{lem_type1_GR}
Suppose that $G$ is a type $1$ partitioned graph. Then $G$ is generically globally $\mL$-rigid.
\end{lemma}

\begin{proof}
We prove this by induction on the number of vertices. 
First observe that a type 1 partitioned graph has at least 5 vertices.
Suppose that $|\mV| = 5$.
So
$\mV = \{v_1,\dots,v_5\}$ and $\mE = \{v_1v_2,v_2v_3,v_3v_1,v_3v_4,v_4v_5,v_5v_3\}$ and note that we can choose the labelling so that 
$\{v_1,v_2\}, \{v_1,v_3\}, \{v_3,v_4\}, \{v_4,v_5\}$ are all crossing. 
Suppose that $p,q \in \mL^\mV$, $p$ is $\mL$-generic and
$M_{G}(q) = M_{G}(p)$. By Theorem~\ref{thm_rigidity_matroid}, $(G-v_1,p|_{\mV-v_1})$ is infinitesimally $\mL$-rigid and so by Lemma~\ref{lem_tech} we see that there is some isometry $\theta$ of $\{L_{v_1},L_{v_2},L_{v_3}\}$ such that $q(x) = \theta p(x)$ for $x \in \{ v_1,v_2,v_3\}$. 
Similarly there is some isometry $\gamma$ of
$\{L_{v_3},L_{v_4},L_{v_5}\}$ such that 
$q(x) = \gamma p(x)$ for $ x \in\{ v_3,v_4,v_5\}$. 
Now, if $\theta = \Id$ then $q(v_3) = p(v_3)$ and since $p$ is $\mL$-generic, 
$\gamma$ restricted to $L_{v_3}$ is the identity.
Hence by Lemma~\ref{lem_isomgroup}, $\gamma = \Id$. 
Therefore $q = p$ in this case. On the other hand if $\theta \neq \Id$ then, again using the fact that $p$ is $\mL$-generic it follows that $\theta|_{L_{v_3}} = \gamma|_{L_{v_3}}$
and, since $\mL$ is in general position, it follows that 
$\{L_{v_1},L_{v_2},L_{v_3}\} = \{L_{v_3},L_{v_4},L_{v_5}\}$ and hence that $\gamma = \theta \in \isom(\mL[G])$. 
Thus $q = \theta p$ for some $\theta\in \isom(\mL[G])$, as required. 

Now suppose that $|\mV| \geq 6$ and let $e,f$ be crossing edges, one in each cycle of $G$ that are both not incident with the vertex of degree 4. Since one of the cycles must have at least 4 vertices, there is 
some vertex $w \in \mV$ of degree 2 that is not incident to either $e$ or $f$. 
Then the graph $G'$ obtained from $G$ by smoothing at $w$ is again a type 1 partition graph. By induction, $G'$ is generically globally $\mL$-rigid, and, by Theorem~\ref{thm_main_henn},
$G$ is generically globally $\mL$-rigid.
\end{proof}

\begin{lemma}
\label{lem_type2_GR}
Suppose that $G$ is a type $2$ partitioned graph. Then $G$ is generically globally $\mL$-rigid.
\end{lemma}

\begin{proof}
Again we proceed by induction on $|\mV|$. 
Suppose that $|\mV| = 6$ (which is minimal for type 2).
Then $\mV = \{v_1,\dots,v_6\}$, $\mE = \{v_1v_2,v_2v_3,v_3v_1,v_4v_5,v_5v_6,v_6v_4,v_1v_4\}$
and we may choose the labelling so that 
$v_1v_2, v_1v_3, v_4v_5, v_5v_6$ are all crossing 
edges. 

Suppose that $(G,p)$, $(G,q)$ are $\mL$-frameworks 
such that $p$ is $\mL$-generic and  $M_{G}(q) = M_{G}(p)$.
By Lemma~\ref{lem_tech} there is some 
$\theta \in  \isom (\{L_{v_1},L_{v_2},L_{v_3}\})$ such that $q(v_3) = \theta p(v_3)$ 
and some 
$\gamma \in  \isom (\{L_{v_4},L_{v_5},L_{v_6}\})$ such that $q(v_4) = \gamma p(v_4)$.

Suppose that $\theta p(v_3) = p(v_3)$. Then $\theta = \Id$ since $p$ is $\mL$-generic. If in addition $\gamma p(v_4) = p(v_4)$ then $\gamma = \Id$ and it follows that $q=p$. 
On the other hand if $\gamma p(v_4) \neq p(v_4)$ then let $y = (\gamma p(v_4) +p(v_4))/2$. Since $y$ is a fixed point of $\gamma$ the coordinates of $y$ lie in $\mathbb Q(\mL)$. However $\|p(v_3) -p(v_4)\| = \|q(v_3) - q(v_4)\| = \|p(v_3) -\gamma p(v_4)\|$. So the line joining $p(v_3)$ and $y$ is perpendicular to $L_{v_4}$ which contradicts the fact that $p$ is $\mL$-generic. Hence we have shown that $ \theta p(v_3) = q(v_3) \neq p(v_3)$ and by symmetry $\gamma p(v_4) = q(v_4) \neq p(v_4)$.

Define
$f(x) = \|x-p(v_3)\|^2 - \|\gamma\cdot x - q(v_3)\|^2$ for $x\in L_{v_4}$. Observe that $\gamma\cdot x = 2z -x$ where $z$ is the unique fixed point of $\gamma$ in $L_{v_4}$ and $q(v_3) = 2w-p(v_3)$ where $w$ is the unique fixed point of $\theta$ in $L_{v_3}$. It follows that $f$ is a polynomial 
function on $L_{v_4}$ with coefficients in $ \mathbb Q(\mL)(p(v_3))$ (i.e. $\mathbb Q(\mL)$ extended by the coordinates of $p(v_3)$). Now $\|p(v_4)-p(v_3)\|=\|q(v_4)-q(v_3)\|$ implies $f(p(v_4)) = 0$. 
Since $p$ is $\mL$-generic,  
it follows that $f(x) = 0$ for all $x \in L_{v_4}$. In particular 
$f(z) = 0$. Therefore $\|z - p(v_3)\| = \|z-q(v_3)\|$. So it follows that either $L_{v_3} = L_{v_4}$ and $w = z$ or that the segment $[w,z]$ is perpendicular to $L_{v_3}$. In either case, since $\mL$ is in general position, it easily follows that $\theta = \gamma\in \isom(\mL[G])$ and so $q = \theta p$ and so $(G,p)$ is globally rigid.

The case when $|\mV|>6$ can be solved in the same manner as in the proof of Lemma~\ref{lem_type1_GR} by induction applying Theorem~\ref{thm_main_henn} at a vertex of degree two.
\end{proof}

\begin{lemma}
\label{lem_type3_GR}
Suppose that $G$ is a type $3$ partitioned graph. Then $G$ is generically globally $\mL$-rigid.
\end{lemma}
\begin{proof}
Again we proceed by induction on $|\mV|$. First suppose that $|\mV| = 4$. Then $\mV = \{u,v,w,x\}$, $\mE = \{uv,vw,wx,xu,uw\}$. 
By the property of type 3, each of the two paths $uvw$ and $uxw$ contains a crossing edge.
Suppose that $p,q \in \mL^\mV$ with $p$ $\mL$-generic and $M_G(q) = M_G(p).$ We split the proof into three cases.

\medskip

Case 1: $uw$ is not a crossing edge.

In this case $uv,vw,wx,xu$ are all crossing. By Lemma~\ref{lem_tech}, there is some $\theta \in \isom(\{L_u,L_v\})$ such that $q(z) = \theta p(z)$ for $z \in \{u,v,w\}$, 
and there is some $\gamma \in \isom (\{L_u,L_x\})$ such that 
$q(z) = \gamma p(z)$ for $z \in \{u,x,w\}$. 
Since $\gamma p(u) = q(u)=\theta p(u)$ we have $\gamma = \theta$.
Since $\mL$ is in general position, either $\theta$ is the identity or 
$\isom(\{L_u,L_v\})=\isom(\{L_u,L_v,L_x\})=\isom(\mL[G])$ holds by Lemma~\ref{lem_isomgroup}.
In either case, $q=\theta p$ for some $\theta\in \isom(\mL[G])$.

\medskip

Case 2: $uw$ is crossing and $|\mL[G]| = 2$. 

Relabelling if necessary, we can assume that $uv,uw,wx$ are crossing edges and that $vw,xu$ are not crossing edges. 
By Lemma~\ref{lem_triangle_flex} 
$q(u) = \theta p(u)$ for some $\theta \in \isom\{L_u,L_w\}$ 
and $q(w) = \gamma p(w)$ for some $\gamma \in \isom\{L_u,L_w\}$. 
If $\theta$ is the identity, then 
we have 
$\|p(u) -p(w)\| = \|q(u) - q(w)\|=\|p(u)-\gamma p(w)\|$,
which implies $\gamma=\Id=\theta$ due to the $\mL$-genericity of $p$. 
Hence, since $|\isom\{L_u,L_w\}|=2$ by Lemma~\ref{lem_isomgroup}, we can always deduce $\gamma=\theta$.
Moreover $\{L_u,L_w\} = \mL[G]$, so $\theta \in \isom(\mL[G])$. Now, replacing $q$ by $\theta^{-1}q$, we see that $q(u) = p(u)$ and $q(w) = p(w)$. Since $p$ is $\mL$-generic, we know that $L_v$ is not perpendicular to $[p(u),p(w)]$ and so $q(v) = p(v)$. Similarly $q(x) = p(x)$ and so $q=p$.

\medskip

Case 3: $uw$ is crossing and $|\mL[G]| \geq 3$.

Relabelling if necessary we can assume that $uv, vw,uw$ are all crossing edges. By Lemma~\ref{lem_tech} there is some $\theta \in \isom(\{L_u,L_v,L_w\})$ such that 
$q(z) = \theta p(z)$ for $z \in \{ u,v,w\}$. But $L_u,L_v,L_w$ are pairwise distinct and so by Lemma~\ref{lem_isomgroup}, $\theta =\Id$ and so $q(z) = p(z)$ for $z \in \{ u,v,w\}$. Now since $p$ is $\mL$-generic, $[p(u),p(w)]$ is not perpendicular to $L_x$ and it follows that $q(x) = p(x)$.
This completes the proof for the case when $\|\mV\|=4$.

\medskip

Now suppose that $|\mV| \geq 5$. 
Let $u,w$ be the vertices of degree $3$ and let $P_1,P_2,P_3$ be the three internally vertex disjoint paths joining $u,v$. 
Since $G$ is type 3, 
we may assume that
$P_1$ contains a crossing edge $e_1$ not incident to $w$
and $P_2$ contains a crossing edge $e_2$ not incident to $u$.
If $P_3$ contains  an internal vertex $x$, 
then smoothing at $x$ results in a graph of type 3.
Hence, by applying the induction hypothesis and then Theorem~\ref{thm_main_henn}, the generic global $\mL$-rigidity of $G$ follows.

Therefore, we may assume $|V(P_3)|=2$. Since $|\mV|\geq 5$, we may further suppose $|V(P_1)|\geq 3$.
Let $y$ be a vertex of $P_1$, which is closest from $u$ among those which are not belong to the same vertex component $V_i$ as that of $u$ in the vertex partition. 
Since $|V(P_1)|\geq 3$, $y$ has a neighbor $z\notin \{u,w\}$.
Observe that  smoothing at $z$ results in a graph of type 3.
Hence, by applying the induction hypothesis and then Theorem~\ref{thm_main_henn}, the generic global $\mL$-rigidity of $G$ follows.
\end{proof}

\begin{remark}
For a given $\mL$, we may define the $\mL$-rigidity matroid of $(G,p)$ to be the matroid, with ground set $\mE$, in which a set $E\subset \mE$ is independent if the rows of $R'(G,p,\mL)$ associated with $E$ are linearly independent. 
It is clear that the rank function of this matroid is determined by which square submatrices of $R'(G,p,\mL)$ have vanishing determinant. 
Thus we may define the {\em generic line-constrained rigidity matroid} to be the $\mL$-rigidity matroid for any set  $\mL$ of lines in general position. 
We note that type 1, 2 or 3 graphs are all circuits in the generic line-constrained matroid. However, there are circuits that are not type 1, 2, or 3. 

We also note that there are circuits that are not globally rigid, in contrast to the classical one dimensional rigidity matroid\footnote{The (generic) $d$-dimensional rigidity matroid may be defined similarly to the (generic) $\mL$-rigidity matroid. The difference is that the matroid is the row matroid of the matrix of coefficients of the linear system in Equation \ref{eqn_inf_bar_constraint} rather than $R'(G,p,\mL)$.}  in which all circuits are also globally rigid. For example, the (generic) framework on the left in Figure \ref{Example 2 (P-connected)} represents a circuit in the line constrained rigidity matroid which is not generically globally rigid. 
\end{remark}

\section{Characterising Global \texorpdfstring{$\mL$}{L}-rigidity}
\label{sec:main}

In this section we give the main theorem of this paper, a combinatorial characterisation of generic global $\mL$-rigidity.
We begin with the case when $G$ is 2-connected.

\begin{lemma}
\label{lem_2conncase}
Suppose that $G$ is P-connected, $2$-connected and generically redundantly $\mL$-rigid.  Then $G$ contains a subgraph of type $3$.
\end{lemma}

\begin{proof} 
Suppose, seeking a contradiction, that $G$ satisfies the hypotheses of the lemma but 
does not contain a type $3$ subgraph.

\begin{claim}
\label{clm_zero}
There are crossing edges $e,f \in E(G)$ that are vertex disjoint. 
\end{claim}

\begin{proof}
Suppose not. Since $G$ is rigid, it has a crossing edge $e$.  By redundant rigidity $ G-e$ is rigid, so there is some crossing edge $f \neq e$. By assumption $e$ and $f$ are not disjoint, so $e = uv$ and $f = vw$ for some vertices $u,v,w$. Now, by $P$-connectivity, $G-v$ has a crossing edge $g$. If $g \neq uw$ then two of $e,f,g$ are vertex disjoint contradicting our assumption. So $g=uw$ is the only crossing edge of $G-v$. Furthermore, if there is some crossing edge $h \neq e,f$ that is incident to $v$ then $h,g$ are disjoint crossing edges, contradicting our assumption.   It follows that $e$ and $f$ are the only crossing edges in $G-uw$ and so $G-uw$ has no crossing cycle, contradicting the assumption that $G-uw$ is rigid.
\end{proof}

\begin{claim}
\label{clm_one}
There are crossing edges $e,f$ and cycles $C,D$ so that $e,f$ are vertex disjoint and $e \in E(C) \setminus E(D)$, $f \in E(D) \setminus E(C)$.
\end{claim}

\begin{proof}
By Claim~\ref{clm_zero} we can choose vertex disjoint crossing edges $e,f\in \mE$. 
Since $G$ is 2-connected there are (not necessarily distinct) cycles $C,D$ such that $e \in E(C)$ and $f \in E(D)$.
Suppose that one of 
$C,D$, without loss of generality $D$, contains both of $e,f$. If $G-\{e,f\}$ is
connected then it follows that $G$ has a subgraph of type 3, contradicting our assumption.
So we can assume that $G -\{e,f\}$ is not connected. 

Now $G-e$ is rigid and so contains a crossing cycle $D'$. Since $f$ is a bridge of $G-e$, $D'$ is contained in $G-\{e,f\}$. Now $D'$ contains at least two crossing edges, $f',f''$ and without loss of generality we can assume that $f'$ and $e$ are vertex disjoint. 
Then $C\cup D'$ contains a cycle $C'$ which contains $e$ and avoids $f'$, and $e,f', C',D'$ are the required edges and cycles.
\end{proof}

Now we can complete the proof of Lemma \ref{lem_2conncase}.
Let $e,f$ be edges and $C, D$ be cycles as in the statement of Claim \ref{clm_one}. Suppose that $e$ is vertex disjoint from $D$. Since $G$ is 2-connected there are two vertex disjoint paths $P_1,P_2$ from the endpoints of $e$ to the vertex set of $D$. Since $D$ contains the crossing edge $f$, the graph  $P_1 \cup P_2 \cup e \cup D$ is a type 3 graph, contradicting our assumption. 

So we may suppose that $e$ is incident to $D$. If $e$ is a chord of $D$ then since $f$ and $e$ are disjoint, the graph $D \cup e$ is type 3. Otherwise suppose $e = uv$ with $v \in V(D)$ and $u \not\in V(D)$. Then, since $G$ is $2$-connected, we can find a path $P$ from $u$ to $D$ that does not contain $v$. Now $P\cup e \cup D$ is a type 3 graph. Thus in all cases we arrive at a contradiction. 

\end{proof}

Now we recall a basic and well-known property of 2-connected graphs. 
We include a proof for completeness. 

\begin{lemma}
\label{lem_ear}
Suppose that $D$ is a $2$-connected graph and that $K$ is a subgraph of $D$ with at least two vertices. There is a sequence of subgraphs $K_1, \dots, K_s$ of $D$ such that $K_1 = K$, $K_s = D$ and for $i= 1,\dots, s-1$,  $K_{i+1} = K_i \cup P_i$ where $P_i$ is a simple path internally vertex disjoint from $K_i$ whose endvertices are distinct and both lie in $K_i$.
\end{lemma}
\begin{proof}
Suppose that we have constructed $K_1,\dots,K_j$ as required and that $K_j \neq D$. 
If $V(K_j) = V(D)$ then choose some edge $uv \in E(D) \setminus E(K_j)$, let $P_j$ be the subgraph of $D$ induced by $\{u,v\}$. 
On the other hand if $V(K_j) \neq V(D)$ then since $D$ is connected we can choose an edge $uv \in E(D)$ such that $u\in V(K_j)$ and $v \not\in V(K_j)$. 
Since $|V(K_j)| \geq |V(K)| \geq 2$ and since $D$ is 2-connected there is a path from $v$ to $V(K_j)$ that misses $u$. 
Let $R$ be a path of minimal length from $v$ to $V(K_j)$ such that $u \not\in V(R)$ and let $P_j = R+u +uv$. 
Now let $K_{j+1} = K_j \cup P_j$ and we can continue in this way until we have constructed the required sequence.
\end{proof}

In the case where $K$ is the graph consisting of a single edge, the sequence $K_1,\dots, K_s$ is often referred to as an open ear decomposition of $D$.
Hence, we call a path $P_i$ given in Lemma~\ref{lem_ear} {\em an open ear} of $K_i$.

\begin{lemma}
\label{lem_2connGR}
Suppose that $G$ is a crossing partitioned graph
and that $\mL$ is a  set of lines in general position. 
If $G$ is P-connected, redundantly $\mL$-rigid and $2$-connected then $G$ is generically globally $\mL$-rigid.
\end{lemma}

\begin{proof}
Let $H$ be the subgraph whose existence is asserted by Lemma~\ref{lem_2conncase}. 
By Lemmas~\ref{lem_type1_GR},~\ref{lem_type2_GR} and~\ref{lem_type3_GR} $H$ is generically globally $\mL$-rigid. Now, by Lemma~\ref{lem_ear} there is a sequence of graphs 
$H = G_1, G_2, \dots, G_s = G$ and for $i = 1,\dots, s-1$ 
such that $G_{i+1}=G_i\cup P_i$, where $P_i$ is an open ear of $G_i$.
Since edge-addition and subdivision preserve generic global $\mL$-rigidity by Theorem~\ref{thm_main_henn}, the addition of an open ear preserves  generic global $\mL$-rigidity. Hence $G$ is generically globally $\mL$-rigid.
\end{proof}

To deal with partitioned graphs that are not $2$-connected we first review some basic facts about block decompositions of graphs. A {\em block} of a graph $G$ is a subgraph that is either a maximal 2-connected subgraph of $G$, or is a copy of $K_2$ whose edge is a bridge of $G$. Let $B(G)$ be the set of blocks of $G$ and $C(G)$ be the set of cutvertices of $G$. The {\em block-cutvertex forest} of $G$ is the bipartite forest whose vertex set is $B(G) \cup C(G)$ and whose edge set is $\{Dv: D \in B(G), v \in C(G)\cap V(D)\} $. Clearly the block-cutvertex forest is a tree if and only if $G$ is connected. A {\em leaf block} is a block that is a 
leaf of the block-cutvertex forest.

\begin{figure}
    \centering
    \begin{tabular}[t]{cc}
\begin{tikzpicture}
    \node[vertex] (a) at (0,0){2};
    \node[vertex] (b) at (-2,0){\phantom{1}};
    \node[vertex] (c) at (-1,-1.5){\phantom{1}};
    \node[vertex] (d) at (-1,1.5){\phantom{1}};
    \node[vertex] (e) at (1,1.5){3};
    \node[vertex] (f) at (2,0){4};
    \node[vertex] (g) at (0,-2){1};
    \node[vertex] (h) at (-1,-3.5){\phantom{1}};
    \node[vertex] (i) at (1,-3.5){\phantom{1}};
    \node[vertex] (j) at (2.5,1.4){\phantom{1}};
    \node[vertex] (k) at (2.2,-1.5){5};
    \node[vertex] (l) at (2,-3){\phantom{1}};
    \node (b1) at (0,-3) {$D_1$};
    \node (b2) at (0.3,-1) {$D_2$};
    \node (b3) at (-1,-.5) {$D_3$};
    \node (b4) at (.6,.7) {$D_4$};
    \node (b5) at (1.9,1.7) {$D_5$};
    \node (b6) at (2.5,-.7) {$D_6$};
    \node (b7) at (2.5,-2.3) {$D_7$};
    \draw[edge] (a) edge (b);
    \draw[edge] (a) edge (c);
    \draw[edge] (b) edge (c);
    \draw[edge] (a) edge (d);
    \draw[edge] (d) edge (e);
    \draw[edge] (e) edge (f);
    \draw[edge] (a) edge (f);
    \draw[edge] (a) edge (g);
    \draw[edge] (h) edge (g);
    \draw[edge] (h) edge (i);
    \draw[edge] (g) edge (i);
    \draw[edge] (e) edge (j);
    \draw[edge] (f) edge (k);
    \draw[edge] (k) edge (l);
\end{tikzpicture}
         &  
\begin{tikzpicture}
\clip (-4.5,-3) rectangle (3.5,3);
    \node[vertex] (1) at (-2,1) {\Large 1};
    \node[vertex] (2) at (-1,1) {\Large 2};
    \node[vertex] (3) at (-0,1) {\Large 3};
    \node[vertex] (4) at (1,1) {\Large 4};
    \node[vertex] (5) at (2,1) {\Large 5};
    \node[vertex] (d1) at (-3,-1) {\footnotesize$D_1$};
    \node[vertex] (d2) at (-2,-1) {\footnotesize$D_2$};
    \node[vertex] (d3) at (-1,-1) {\footnotesize$D_3$};
    \node[vertex] (d4) at (0,-1) {\footnotesize$D_4$};
    \node[vertex] (d5) at (1,-1) {\footnotesize$D_5$};
    \node[vertex] (d6) at (2,-1) {\footnotesize$D_6$};
    \node[vertex] (d7) at (3,-1) {\footnotesize$D_7$};
    \draw[edge] (1) -- (d1);
    \draw[edge] (1) -- (d2);
    \draw[edge] (2) -- (d2);
    \draw[edge] (2) -- (d3);
    \draw[edge] (2) -- (d4);
    \draw[edge] (3) -- (d4);
    \draw[edge] (3) -- (d5);
    \draw[edge] (4) -- (d4);
    \draw[edge] (4) -- (d6);
    \draw[edge] (5) -- (d6);
    \draw[edge] (5) -- (d7);
\end{tikzpicture}

    \end{tabular}
    \caption{The graph on the left has cutvertices labelled $1,\cdots,5$ and 
    blocks labelled $D_1,\cdots,D_7$. The corresponding block-cutvertex forest is shown on the right. The unique path connecting the blocks $1$ and $7$ is $D_1,1, D_2, 2, D_4, 4, D_6, 5, D_7$. So $[D_1,D_7] =
    D_1 \cup D_2 \cup D_4 \cup D_6 \cup D_7$.} 
    \label{Block}
\end{figure}
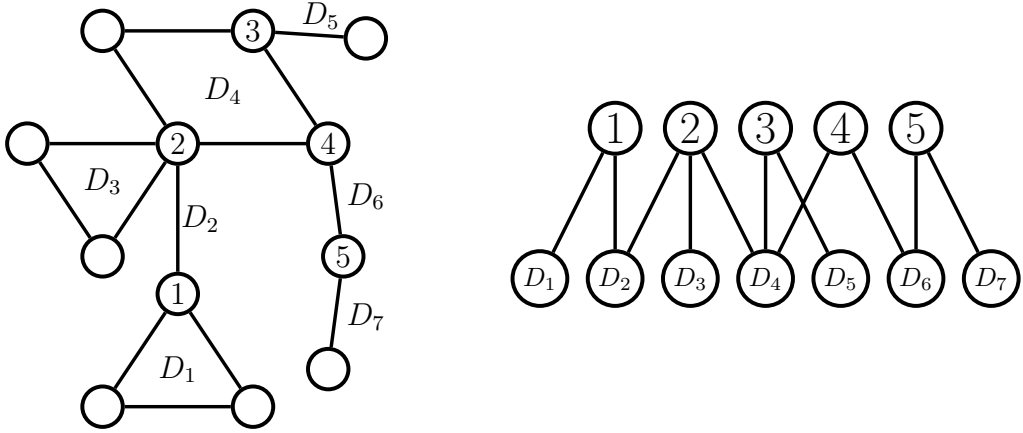

Now suppose that $L$ and $K$ are distinct leaf blocks of a connected graph. Then there is a unique simple path $L= D_1,v_1,D_2,\dots,D_{k-1},v_{k-1}, D_k=K$ in the block-cutvertex tree. (See Figure \ref{Block} for an illustration of these concepts.)
Let $[L,K] = D_1\cup \dots \cup D_k$. 

\begin{lemma}
\label{lem_btpathGR}
Suppose that $G$ is connected, P-connected and redundantly rigid and that $L,K$ are distinct leaf blocks of $G$. Then $[L,K]$ is globally $\mL$-rigid.
\end{lemma}
\begin{proof}
Suppose that $[L,K] = D_1 \cup \cdots \cup D_k$ where 
$L = D_1,\dots, D_k= K$ are blocks of $G$. Let $v_i = V(D_i) \cap V(D_{i+1})$, $ i =1,\dots, k-1$.
By redundant rigidity, $K$ and $L$ are maximal 2-connected subgraphs of $G$.
Since $G$ is P-connected there is a crossing edge 
$e$ in $L-v_1$.
Since $L$ is 2-connected, there is a cycle $C_L$ in $L$ such that $e \in E(C_L)$. 
Similarly, $K$ has a cycle $C_K$ that contains a crossing edge not incident to $v_{k-1}$.

By the connectivity of $G$, $G$ has a path $P$ with endvertices in
$V(C_L)$ and $V(C_K)$ such that $P$ is internally disjoint from $C_L$ and $C_K$.
Let $H=C_L\cup C_K\cup P$. 
Then $H$ is a type 1 or type 2 partitioned graph,
and hence it is generically globally $\mL$-rigid by Lemmas~\ref{lem_type1_GR} and \ref{lem_type2_GR}.

We now augment $H$ to $[L,K]$ by adding open ears sequentially. 
Specifically, for each $i=1,\dots, k-1$, 
$E(D_i)\cap E(H)\neq \emptyset$ holds by the definition of block-cutvertex forests.
Hence, by Lemma~\ref{lem_ear}, $D_i\cap H$ can be augmented to $D_i$ by adding open ears.
Therefore, $H$ can be augmented to $[L,K]$ by adding open ears.
Since the addition of an open ear preserves generic global $\mL$-rigidity by Theorem~\ref{thm_main_henn}, the global $\mL$-rigidity of $[L,K]$ follows from that of $H$.
\end{proof}

Finally, our main result. 

\begin{theorem}
\label{thm_main}
Let $\mL$ be a set of lines in general position.
A crossing partitioned graph $G$ is generically 
globally $\mL$-rigid if and only if $G$ is 
P-connected and redundantly $\mL$-rigid.
\end{theorem}

\begin{proof}
Following Theorem~\ref{thm_nec}, it remains to show that if $G$ is P-connected and redundantly $\mL$-rigid then it is generically globally $\mL$-rigid. 

If $G$ is 2-connected then Lemma~\ref{lem_2connGR} gives the required conclusion. 
If $G$ is not 
2-connected but is connected, then there are leaf blocks $L_i,K_i$, $i =1,\dots,s$
such that $L_i \neq K_i$ for $i = 1,\dots,s$, $ G = \cup _{i=1}^s [L_i,K_i]$ and such that for $i =2,\dots,s$, $[L_i,K_i]$ shares a block with $[L_j,K_j]$ for some $j <i$. 
Now an easy induction argument using Lemma~\ref{lem_btpathGR} and Theorem~\ref{thm_gluing} shows that $G$ is generically globally $\mL$-rigid.

If $G$ is not connected, then the $P$-connectivity of $G$ implies that $|\mL[H]|\geq 3$ for each connected component $H$ of $G$. 
Hence, the generic global $\mL$-rigidity of $G$ follows by applying the induction hypothesis to each connected component and then applying Theorem~\ref{thm_gluing}.
\end{proof}

\begin{corollary}
Suppose that $(G,p)$ is a generic globally rigid $\mL$-framework. Then $G$ is generically globally $\mL$-rigid.
\end{corollary}

\begin{proof}
By Lemma~\ref{lem_Pconn} and Theorem~\ref{thm:hendrickson}, $G$ is P-connected and redundantly $\mL$-rigid. The result follows from Theorem~\ref{thm_main}.
\end{proof}

\section{Concluding Remarks}

\noindent 1. Theorems \ref{thm:rigid} and \ref{thm_main} are good characterisations in the sense that they lead quickly to efficient deterministic algorithms to check $\mL$-rigidity and global $\mL$-rigidity. 
\newline

\noindent 2. It would be natural to consider extensions of our results where lines are replaced by (two-dimensional) planes. When $\mL$ is replaced by a set of parallel planes then characterising rigidity is a straightforward extension of the plane case (see e.g.~\cite[Theorem 5.1]{NOP}) and 
global rigidity is similarly easy. 
However allowing the planes to be non-parallel, even for generic planes, opens up substantial difficulties that arise in the standard bar-joint rigidity model in dimension at least 3. One fundamental such difficulty is the existence of `flexible circuits' in the $d$-dimensional rigidity matroid when $d\geq 3$ (see \cite{GGJN,JJbounded} inter alia). Perhaps the simplest non-trivial flexible circuit that can arise in the context under discussion is the graph obtained from two copies of $K_4$ sharing a single vertex $v$ such that $v$ is the only vertex on the second plane. Here the natural sparsity counts would predict rigidity but there is an obvious motion.
Similarly, extending the 2-dimensional global rigidity characterisation of Jackson and Jord{\'a}n~\cite{J&J} to non-parallel planes is a challenging open problem.
\newline

\noindent 3. A related generalisation of our results would be to allow one additional part in the partition of a partitioned graph, with this part corresponding to vertices that are not constrained to some line of $\mL$. To avoid well known problems in dimension at least 3, let us restrict to the case when $d=2$. When both the set of points and the set of lines are generic then rigidity \cite{ST} and global rigidity \cite{GJN} are understood. When the points are generic but the lines are allowed to be non-generic then a result of \cite{Katoh-Tanigawa} applies for rigidity. When more than two vertices are allowed on any given line then, even with generic lines, this seems to be an open but potentially tractable problem for both rigidity and global rigidity.
\newline

\noindent 4. It is possible to adapt Theorem \ref{thm:rigid} to apply to `circle constrained frameworks'; here the set of lines is replaced by a set of circles and the analogue of parallel is concentric. 
The intuition for this translation is that the constraint that a point $p(v)$ moves on a circle $C$ at the infinitesimal level is simply the constraint that $p(v)$ moves on the tangent line to $C$ at $p(v)$. However this observation does not provide any information about whether the global rigidity problem is equivalent in the circle constrained rigidity model. In \cite{GJN} an equivalence for global rigidity was deduced in dimension 2 in the special case where the set of points and lines is generic. It may also be interesting to extend this question to general curves.

\subsection*{Acknowledgements}
This project grew from discussions at the Fields Institute Thematic Program on Geometric Constraint Systems, Framework Rigidity, and Distance Geometry, and we are grateful to the organizers for bringing us together. 
F.M. was partially supported by the KU Leuven grant iBOF/23/064, the UiT Aurora project MASCOT, and the FWO grants
G0F5921N and G023721N. 
A.N. was partially supported by EPSRC grant number EP/W019698/1.
S.T. was partially supported by JST PRESTO Grant Number JPMJPR2126 and JSPS KAKENHI Grant Number 20H05961.	

\bibliographystyle{abbrv}

\bigskip 

\noindent
\footnotesize {\bf Authors' addresses:}

\medskip

\noindent{Mathematics, Statistics \& Applied Mathematics, NUI Galway, Ireland
\hfill {\tt james.cruickshank@universityofgalway.ie}
}

\noindent{Department of Mathematics and Department of Computer Science, KU Leuven, Belgium 
\\
Department of Mathematics \& Statistics, 
  University of Troms\o, Norway} \hfill {\tt fatemeh.mohammadi@kuleuven.be}

\noindent{Department of Mathematics, Ghent University, Belgium}
  \hfill {\tt harshitjitendra.motwani@ugent.be}

\noindent{Mathematics and Statistics, 
Lancaster
University, Lancaster,
UK
\hfill {\tt a.nixon@lancaster.ac.uk}
}

\noindent{Department of Mathematical Informatics, University of Tokyo, Japan}  \hfill {\tt tanigawa@mist.i.u-tokyo.ac.jp}
\end{document}